\documentclass[11pt]{article}
\usepackage[centertags,intlimits]{amsmath}
\usepackage{amsfonts}
\usepackage{amsthm}
\usepackage{graphics}
\usepackage{color}
\usepackage{graphicx}
\usepackage{amssymb}
\allowdisplaybreaks[2]
\numberwithin{equation}{section}
\newtheorem{theorem}{Theorem}[section]
\newtheorem{lemma}{Lemma}[section]

\newtheorem{remark}{Remark}[section]

\providecommand{\abs}[1]{\lvert #1\rvert}

\newcounter{neweqn}
\newcommand{\beq}[1]{\addtocounter{neweqn}{1}\begin{equation}\label{#1}}
\newcommand{\eeq}{\end{equation}}
\newcommand{\nr}[1]{(\ref{#1})}
\newcommand{\nc}{\newcommand}
\newcommand{\ra}{\rightarrow}
\newcommand{\beray}[1]{\addtocounter{neweqn}{1}\begin{eqnarray}  \label{#1}}

\nc{\vb}{\mathbf{v}}
\nc{\bx}{\mathbf{x}}
\nc{\by}{\mathbf{y}}
\nc{\bz}{\mathbf{z}}
\nc{\bu}{\mathbf{u}}
\nc{\bv}{\mathbf{v}}
\nc{\ba}{\mathbf{a}}
\nc{\bs}{\mathbf{s}}
\nc{\bq}{\mathbf{q}}
\nc{\bd}{\mathbf{d}}
\nc{\bb}{\mathbf{b}}
\nc{\bc}{\mathbf{c}}
\nc{\bi}{\mathbf{i}}
\nc{\bfr}{\mathbf{r}}
\nc{\bP}{\mathbf{P}}
\nc{\bQ}{\mathbf{Q}}
\nc{\R}{\mathbb R}
\nc{\N}{\mathbb N}
\nc{\bbC}{\mathbb C}
\nc{\D}{\mathbb D}
\nc{\Z}{\mathbb Z}
\nc{\F}{\mathbf F}
\nc{\bbS}{\mathbb S}
\nc{\bE}{\mathbf E}
\nc{\B}{\cal B}
\nc{\br}{\bigr}
\nc{\bl}{\bigl}
\nc{\Bl}{\Bigl}
\nc{\Br}{\Bigr}
\nc{\ind}[1]{\,\mathbf{1}_{\{#1\}}\,}

\author{Anatolii A. Puhalskii
}

\title{Moderate deviations of many--server queues \\
 in 
the Halfin--Whitt regime\\ and weak convergence methods
}

\begin{document}
\maketitle
\sloppy

\section{Introduction}
\label{sec:introduction}

Many--server queues are important in applications but their analysis
beyond Markovian assumptions is difficult, see, e.g., Asmussen
\cite{Asm03}. As a way to gain insight into the operation of
many--server queues,
various heavy--traffic 
asymptotics have been explored when the arrival and service rates tend
to infinity.
Iglehart  and Whitt \cite{IglWhi70}
 considered the number of servers, $n$\,,
 fixed, assumed that the traffic intensity, $\rho$\,, tends to unity 
  and  showed that properly normalized version of the queue
length  process converges to a one dimensional
reflected Brownian motion, in analogy with the heavy--traffic limit
for a $GI/GI/1$  queue.
As for the single server queue, the
probability of having at least one customer in the queue
converges to unity.
Many server queues with the numbers of servers going to infinity
 for different
initial conditions and assumptions on the interarrival and service
time distributions, were considered 
by Borovkov \cite{Bor67q,Bor80}, Iglehart  \cite{Igl65,Igl73} and
Whitt \cite{Whi82}. 
 The growth conditions imposed  on $n$ result in the queue being
asymptotically equivalent to an infinite server queue
so that the probability of wait tends to zero.

In order  to capture the situation of
the probability of a customer wait being strictly between zero and
unity,
 Halfin and
Whitt \cite{HalWhi81} proposed the setup where
   service time distributions are held fixed, while
 the number of servers and the arrival rate
 grow without a bound in such a way that 
$\sqrt{n}(1-\rho)\to\beta\in\R\,.$
With  $Q(t)$ denoting the number of customers present at
time  $t$\,, assuming the initial conditions are suitably chosen, the
sequence of processes $(Q(t)-n)/\sqrt{n}$\,,
considered as  random elements of the associated Skorohod space,
converges in law to a continuous--path process. 
In their  pioneering contribution,
   Halfin and Whitt \cite{HalWhi81} considered the case of
 exponentially distributed service times. The limit process then is an Ito
 diffusion with a piecewise linear drift.
 An extension  to generally
 distributed service times is more recent and is due to Reed
 \cite{Ree09}, see also Puhalskii and Reed \cite{PuhRee10}. 
The limit process, however, is no longer generally Markov and is not
well understood, cf. Kaspi and Ramanan \cite{KasRam13} and
Aghajani and Ramanan \cite{AghRam20}, for related results.

In search of more tractable asymptotics, it was proposed in 
Puhalskii \cite{Puh22} to study moderate deviations of $Q(t)$ and a Large
Deviation Principle (LDP) was conjectured  for the process
$(Q(t)-n)/(b_n\sqrt{n})$ under the heavy traffic condition
$  \sqrt{n}/b_n(1-\rho)\to\beta$\,,
where $b_n\to\infty$ and $b_n/n\to0$\,. 
 The 
deviation function (a.k.a. rate function) was purported to solve
a convex variational problem with a quadratic objective function. 
This paper confirms  the conjecture and proves the LDP in question. 
Furthermore,  the deviation function  is expressed in terms of
the solution to
   a Fredholm
equation of the second kind which could facilitate
 evaluating the deviation
function  numerically.
It is  noteworthy that although the LDP generally produces cruder
asymptotics than weak convergence, or convergence in law, it may capture  exponents in the
distributions of corresponding weak  convergence limits, cf. 
Puhalskii \cite{Puh99a}.

  The main thrust of the proofs  is to demonstrate the utility of
weak convergence methods for deriving large deviation asymptotics. As in Puhalskii \cite{puhwolf,Puh01}, the LDP is viewed
as an analogue of weak convergence, the cornerstone of the approach 
being the following analogue of celebrated Prohorov's tightness theorem: ``A
sequence of probability measures on a complete separable metric space
is exponentially tight if and only if every its  subsequence admits a further
subsequence that satisfies the LDP'' (Theorem (P) in Puhalskii
\cite{Puh91}). 
Consequently, once exponential
tightness has been proved, the
proof of the LDP is achieved by identifying 
a  large deviation limit point. A possible identification tool is an
analogue of the method of finite--dimensional distributions.
It is what is used in
 this paper in order to
 derive
 logarithmic asymptotics of
moderate deviations  of
the sequential empirical process, which  is a key element of the
proof of the main result, on the one hand,
 and  complements developments in Krichagina and Puhalskii
\cite{kripuh97} and in Puhalskii and Reed \cite{PuhRee10}, on the
other hand. In 
 order to use the analogy with weak convergence  to the utmost,
this paper formulates the definition of the LDP for stochastic
processes  as
  large deviation (LD) convergence  to
idempotent processes, see Puhalskii \cite{Puh01} and Appendix
\ref{sec:large-devi-conv}  for
definitions and properties. With tools for the study of
 weak convergence properties of  many--server queues 
in heavy traffic being well developed,
 the moderate deviation asymptotics are derived
 by using similar methods, mutatis mutandis,
 that proved their worth in the weak convergence
context. An important distinction on the technical side
is an extensive use of exponential martingales, as
establishing exponential tightness 
is a more challenging  task than  establishing
 tightness.
In addition,   a study of
 idempotent counterparts of
the standard Wiener process, the Brownian bridge and the
Kiefer process is carried out. The properties of those idempotent
processes are integral to the proofs.

The paper's organisation 
is as follows.
In Section
\ref{sec:fluid-stoch-appr}, a precise specification of
 the model is provided, as well as the main result on the  logarithmic
asymptotics of moderate deviations of the number--in--the--system
process. 
 The latter
 is stated both in the
form of a trajectorial 
LDP and in the form of LD convergence to an idempotent process.
The proofs are given in Section 
\ref{sec:proof-heavy}.
Section \ref{sec:eval-devi-funct} is concerned with evaluating the
deviation function by  reduction to solving a  Fredholm
equation of the second kind.
 For the reader's convenience,
  Appendix \ref{sec:large-devi-conv}
 gives a primer on idempotent
processes and the use of weak convergence methods for proving the
LDP. 
Appendix  \ref{renewal}  is concerned with  absolute
continuity of the solution to a
nonlinear renewal equation which is needed in Section \ref{sec:eval-devi-funct}.

\section{Trajectorial moderate deviation limit theorems}
\label{sec:fluid-stoch-appr}
Assume as  given a sequence of many server queues with unlimited
waiting room indexed by $n$, where $n$
represents the number of servers. Service is performed on a first--come--first--served
basis. Let $Q_n(t)$ denote the number of customers present at time
$t$\,. Out of those, $Q_n(t)\wedge n $
customers are in service and $(Q_n(t)-n)^+$ customers are
in the queue.  
 The service times of the customers in the queue at
time  $0$ and the service times
of customers exogenously arriving after time  $0$ are denoted
 by
 $\eta_1,\eta_2,\ldots$ in  the
 order of  entering service and  come from an
i.i.d. sequence of positive unbounded random variables with
continuous distribution function
$F$\,.  In terms of $F$\,, it is thus assumed that
\begin{equation}
  \label{eq:45}
F(0)=0\,, F(x)<1\,, \text{ for all }x\,.
\end{equation}
The mean service time $\mu^{-1}=\int_0^\infty x\,dF(x)$ is assumed to
be finite. Residual service times of the customers in service at
time  $0$ are denoted by
$\eta^{(0)}_1,\eta^{(0)}_2,\ldots$ and are assumed to be
 i.i.d. with distribution  $F_0$ that is the
 distribution  of the delay in a stationary renewal
 process with inter--renewal distribution $F$\,. Thus,
\begin{equation}
  \label{eq:73}
  F_0(x)=\mu\int_0^x(1-F(y))\,dy\,.
\end{equation}

Let $A_n(t)$ denote the number of exogenous arrivals
by $t$\,, with $A_n(0)=0$\,. It is assumed that $A_n(t)$ has unit jumps.
  The entities   $Q_n(0)$,
$\{\eta^{(0)}_1,\eta^{(0)}_2,\ldots\}$,
$\{\eta_1,\eta_2,\ldots\}$, and $A_n=(A_n(t),\,t\in\R_+)$ are assumed to be
 independent.
All stochastic processes are assumed to have
 rightcontinuous paths with lefthand limits.
Let 
${{\hat A}_n}(t)$ denote the number of customers that enter service after time
 $0$ and by time 
$t$\,, with 
${{\hat A}_n}(0)=0$\,.  Since the distribution function
 of  $\eta_i$ is continuous, $\hat A_n$ has
unit jumps a.s.
 By considering the jumps of the processes on both sides of the next
equation, one can
see that
\begin{equation}
    \label{eq:40}
 (Q_n(t)-n)^++{\hat A}_n(t)=(Q_n(0)-n)^++A_n(t)\,.
\end{equation}
In addition, balancing the arrivals and departures yields the equation
\begin{equation}
             \label{eq:107}
    Q_n(t)=Q_n^{(0)}(t)+(Q_n(0)-n)^+
+A_n(t)-\int_0^t\int_0^t \mathbf{1}_{\{s+x\le t\}}
\,d\,\sum_{i=1}^{{{\hat A}_n}(s)}\mathbf{1}_{\{\eta_i\le x\}}\,,
\end{equation}
where
\begin{equation}
  \label{eq:1a}
Q_n^{(0)}(t)=  \sum_{i=1}^{Q_n(0)\wedge n}\,\ind{\eta^{(0)}_i> t}\,,
\end{equation}
which quantity represents the number of customers present at time $t$ out
of those in service at time $0$\,, and 
\begin{equation}
  \label{eq:79}
  \int_0^t\int_0^t \mathbf{1}_{\{s+x\le t\}}
\,d\,\sum_{i=1}^{{{\hat A}_n}(s)}\mathbf{1}_{\{\eta_i\le x\}}=
\sum_{i=1}^{\hat A_n(t)}\ind{\hat\tau_{n,i}+\eta_i\le t}
\end{equation}
which quantity represents the number of customers that  enter service after  time $0$
and  depart by time $t$\,, $\hat\tau_{n,i}$ denoting the $i$th jump
time of $\hat A_n$\,, i.e.,
\begin{equation}
  \label{eq:24}
  \hat\tau_{n,i}=\inf\{t:\,\hat A_n(t)\ge i\}\,.
\end{equation}
For  existence and uniqueness of solutions to
\eqref{eq:40}--\eqref{eq:1a}, the reader is referred to
 Puhalskii and Reed
\cite{PuhRee10}.

Given a sequence $r_n\to\infty$\,, as $n\to\infty$\,,  a sequence $P_n$
of probability laws on the Borel $\sigma$--algebra of
 a metric space $M$ and  a $[0,\infty]$--valued function $I$
 on $M$ such that the sets $\{y\in M:\,I(y)\le \gamma\}$ are compact
 for all $\gamma\ge0$\,, the sequence $P_n$ is said to obey 
the LDP for rate $r_n$ with  deviation function $I$ provided
$  \lim_{n\to\infty}1/r_n\,\ln P_n(W)=-\inf_{y\in W}I(y)\,,
$ for every Borel set $W$ such that the infima of $I$ over the
interior  and  closure of $W$ agree.

The       deviation function for the queue length process is introduced next.
For $T>0$ and $m\in\N$\,, let $\D([0,T],\R^m)$ and
 $\D(\R_+,\R^m)$ represent the respective Skorohod spaces of rightcontinuous
$\R^m$--valued 
functions  with lefthand limits defined on $[0,T]$ and $\R_+$\,,
respectively.  
These spaces are
 endowed with 
 metrics rendering them  complete separable metric spaces, see Ethier and
 Kurtz \cite{EthKur86}, Jacod and Shiryaev \cite{jacshir}, for
 more detail.
Given $q=(q(t)\,,t\in\R_+)\in\D(\R_+,\R)$\,, let
\begin{equation}
   \label{eq:2}
         I^Q(q)=\frac{1}{2}\inf\{\int_0^1\dot w^0(x)^2\,dx
+\int_0^\infty \dot w(t)^2\,dt
+\int_0^\infty\int_0^1 \dot k(x,t)^2\,dx\,dt\}\,,
 \end{equation}
the $\inf$ being taken over $w^0=(w^0(x)\,,x\in[0,1])\in \D([0,1],\R),
w=(w(t)\,,t\in\R_+)\in\D(\R_+,\R)$ and
$k=((k(x,t)\,,x\in[0,1])\,,t\in\R_+)\in\D(\R_+,\D([0,1],\R_+))$ such that
$w^0(0)=w^0(1)=0$\,, $w(0)=0$\,, $k(x,0)=k(0,t)=k(1,t)=0$\,,
$w^0$, $w$ and $k$ are absolutely continuous with respect to Lebesgue measures
on  $[0,1]$\,, $\R_+$ and $[0,1]\times \R_+$\,, respectively, and, for
all $t$\,, with $q_0^-=(-q_0)^+$\,,
\begin{multline}\label{eq:1}
q(t)=(1-F(t))q_0^+-(1-F_0(t))q_0^-
-\beta F_0(t) +\int_0^tq(t-s)^+\,dF(s)+w^0(F_0(t))\\
+\int_0^t\bl(1-F(t-s)\br)\sigma\,\dot w(s)\,ds
+\int_{\R_+^2} \ind{s+x\le t}\,\dot k(F(x),\mu s)\,dF(x)\,\mu \,ds
\,,
\end{multline}
 where 
 $\dot w^0(x)$\,, $\dot w(t)$ and $\dot k(x,t)$
 represent the respective Radon--Nikodym derivatives.
If  $w^0$\,, $w$ and $k$ as indicated do not exist, then
$I^Q(q)=\infty$\,.
Consequently, $I^Q(q)=\infty$ unless $q(t)$ is  continuous as the
righthand side of \eqref{eq:1} is a continuous function of $t$\,.
(It is proved in Lemma \ref{le:ac} that if $F$ is, in addition, absolutely
continuous with respect to Lebesgue measure, then 
$q(t)$  in \eqref{eq:1} is absolutely continuous too.)
 By Lemma B.1 in
Puhalskii and Reed \cite{PuhRee10}, the equation 
in \eqref{eq:1} has a unique solution for $q(t)$
as an element of $\mathbb L_{\infty,\text{loc}}(\R_+,\R)$\,.

Let the  processes $X_n=(X_n(t),t\in\R_+)$ be defined  by
\begin{equation}
    \label{eq:30}
X_n(t)=\frac{\sqrt{n}}{b_n}\bl(\frac{Q_n(t)}{n}-1\br)\,.
\end{equation}
The    next theorem  confirms a conjecture in Puhalskii
\cite{Puh22}. 
\begin{theorem}
\label{the:ldp}Suppose, in addition,  that
 $A_n$ is a  renewal process of rate $\lambda_n$\,.
Let $\rho_n=\lambda_n/(n\mu)$\,, 
$\beta\in\R$\,, $q_0\in\R$ and $\sigma>0$\,.
 Suppose that,  
as $n\to\infty$,
\begin{equation}
  \label{eq:105}
  \frac{\sqrt{n}}{b_n}(1-\rho_n)\to\beta
\end{equation}
 and
the sequence of random variables
$\sqrt{n}/b_n\,(Q_n(0)/n-1)$ obeys the LDP in $\R$ for rate $b_n^2$
with deviation function
$I_{q_0}(y)$ such that $I_{q_0}(q_0)=0$ and
$I_{q_0}(y)=\infty$\,,  for $y\not=q_0$\,.
Suppose that
 the sequence of processes
$\bl(\sqrt{n}/b_n(A_n(t)/n-
\mu\rho_nt),t\in\R_+\br)
$
obeys the LDP in $\D(\R_+,\R)$  for rate $b_n^2$
with deviation function 
$I^A(a)$ such that $I^A(a)=1/(2\sigma^2)\int_0^\infty\dot a(t)^2\,dt$\,, 
provided $a=(a(t)\,,t\in\R_+)$ is an absolutely continuous nondecreasing function with
$a(0)=0$\,, and
$I^A(a)=\infty$\,, otherwise.\
If, in addition,
\begin{equation}
  \label{eq:3}
  b_n^3n^{1/b_n^2-1/2}\to0\,,
\end{equation}
 then the sequence 
$X_n$ obeys
 the LDP in $\D(\R_+,\R)$ for  rate $b_n^2$
 with deviation function 
$I^Q(q)$\,. 
\end{theorem}
\begin{remark}
  In order that the  LDP for $\bl(\sqrt{n}/b_n(A_n(t)/n-
\mu \rho_nt),t\ge0\br)
$ in the statement hold, it  suffices that $E(n\xi_n)\to1/\mu$\,,
$\text{Var}(n\xi_n)\to \sigma^2/\mu^3$ and that either
$\sup_nE(n\xi_n)^{2+\epsilon}<\infty$\,, for some $\epsilon>0$\,,
and $\sqrt{\ln n}/b_n\to\infty$ or
$\sup_nE\exp(\alpha (n\xi_n)^{\delta})<\infty$ and
$n^{\delta/2}/b_n^{2-\delta}\to\infty$\,, for some $\alpha>0$ and
$\delta\in(0,1]$\,, where $\xi_n$ represents a generic interarrival time
for the $n$-th queue, see Puhalskii \cite{Puh99a}.
\end{remark}

LD convergence versions come next. Introduce
\begin{equation}
  \label{eq:46}
  Y_n(t)=\frac{\sqrt{n}}{b_n}\,\bl(\frac{A_n(t)}{n}-\mu t\br)
\end{equation}
and let
$Y_n=(Y_n(t),\,t\in\R_+)$\,. For the statement of the next theorem,
Appendix \ref{sec:large-devi-conv} is a recommended reading.
\begin{theorem}
  \label{the:heavy_traffic}
  \begin{enumerate}
  \item 
Suppose that,  as $n\to\infty$,
  the sequence $X_n(0)$ LD converges in distribution
in $\R$ at rate $b_n^2$ to an idempotent variable $X(0)$\,,
   the sequence $Y_n$ LD  converges in
 distribution   in $\D(\R_+,\R)$ at rate $b_n^2$ to
an idempotent  process
 $Y$ with continuous paths and \eqref{eq:3} holds.
Then  the sequence $X_n$
LD converges in distribution in $\D(\R_+,\R)$ at rate $b_n^2$ to 
    the idempotent process
$X=(X(t),t\in\R_+)$ that is the unique  solution to  the equation
\begin{multline}
  \label{eq:70}
         X(t)=(1-F(t))X(0)^+-
         (1-F_0(t))X(0)^-+\int_0^tX(t-s)^+\,dF(s)+W^0(F_0(t))
\\
+Y(t)-\int_0^t Y(t-s)\,dF(s)
+\int_0^t\int_{0}^t \ind{s+x\le t}\,\dot K(  F(x),\mu s)\,dF(x)\,\mu\,ds
\,,
 \end{multline}
where
$W^0=(W^0(x),x\in[0,1])$ is a  Brownian bridge idempotent process and
 $K=(K(x,t),
(x,t)\in\R_+\times [0,1])$ is a  Kiefer idempotent process,  $X(0)$,
$Y$\,, $W^0$ and
$K$ being
independent.
\item
In the special case where $A_n$ is a renewal process of rate $\lambda_n$\,,
condition \eqref{eq:105} holds 
 and the sequence
$(\sqrt{n}/b_n\,(A_n(t)-\mu\rho_nt)\,,t\in\R_+)$ LD converges in
distribution in $\D(\R_+,\R)$
at rate $b_n^2$ to $\sigma W$\,, with 
  $W=(W(t)\,,t\in\R_+)$ being a standard Wiener idempotent process,
 the 
 limit idempotent process $X$ solves the equation
 \begin{multline}
   \label{eq:36}
       X(t)=(1-F(t))X(0)^+-(1-F_0(t))X(0)^--\beta F_0(t)
+\int_0^tX(t-s)^+\,dF(s)
\\+W^0(F_0(t))+\int_0^t\bl(1-F(t-s)\br)\sigma\,\dot W(s)\,ds
+\int_0^t\int_{0}^t \ind{s+x\le t}\,\dot K(F(x),\mu s)\,dF(x)\,\mu \,ds
\,,
 \end{multline}
where  $X(0)$,
$W$, $W^0$ and
$K$ are
independent.
  \end{enumerate}

\end{theorem}
It is noteworthy that
 the LD convergence of the sequence $X_n$ in part 2 
  is equivalent to the assertion of Theorem
\ref{the:ldp}\,.
Besides, the limit idempotent process in \eqref{eq:70} is analogous to
the limit stochastic  process
 in  Puhalskii and Reed \cite[p.139]{PuhRee10}.
\section{Proofs of Theorems  \ref{the:ldp} and \ref{the:heavy_traffic}  }
\label{sec:proof-heavy}
This section is mostly concerned with the proof of 
part 1 of Theorem~\ref{the:heavy_traffic}. Part 2 and Theorem
\ref{the:ldp} are easy consequences.  A proof outline is as follows.

Let
\begin{align}
    \label{eq:16}
  H_n(t)&=Y_n(t)-\int_0^t
Y_n(t-s)\,dF(s)\,,\\ \label{eq:63}
 X^{(0)}_n(t)&=\frac{\sqrt{n}}{b_n}\,\bl(\frac{1}{n}\, Q_n^{(0)}(t)
-(1-F_0(t))\br)\,,\\
  \label{eq:33}
    U_n(x,t)&=\frac{1}{b_n\sqrt{n}}\,\sum_{i=1}^{{{\hat A}_n}(t)}
\bl(\ind{\eta_i\le x}-F(x)\br)
\intertext{ and }
  \label{eq:19}
  \Theta_n(t)&=-
\int_{\R_+^2} \mathbf{1}_{\{s+x\le t\}}\,dU_n(x,s)\,.
\end{align}
As, owing to \eqref{eq:73}, 
\begin{equation}
  \label{eq:78}
  1-F_0(t)+\mu t-\mu\int_0^t(t-s)dF(s)=1\,,
\end{equation}
 \eqref{eq:107} implies that
  \begin{align}
      \label{eq:67}
  X_n(t)&=(1-F(t))X_n(0)^++ X_n^{(0)}(t)+
\int_0^tX_n(t-s)^+\,dF(s)
+H_n(t)+\Theta_n(t)
\,.
\end{align}
The equation in  \eqref{eq:70} is written in a similar way: introducing
\begin{align}
          \label{eq:47}
  H(t)&=Y(t)-\int_0^t Y(t-s)\,dF(s)\,,
\\
\label{eq:47a}X^{(0)}(t)&=W^0(F_0(t))-(1-F_0(t))X(0)^-\,,
\\\label{eq:26}
U(x,t)&=K(F(x),\mu t)\intertext{and}
\label{eq:12}  \Theta(t)&=
-\int_{\R_+^2} \mathbf{1}_{\{s+x\le t\}}\,dU(x,s)
=-\int_{\R_+^2} \mathbf{1}_{\{s+x\le t\}}\,\dot K(F(x),\mu s)\,dF(x)\,\mu ds
\end{align}
yield
\begin{equation}
  \label{eq:71}
           X(t)=(1-F(t))X(0)^++X^{(0)}(t)+\int_0^tX(t-s)^+\,dF(s)
+H(t)+\Theta(t)\,.
\end{equation}
It will be proved that the equations \eqref{eq:67} LD converge
termwise to the equation \eqref{eq:71}, as $n\to\infty$\,.
More specifically, the following  theorem will be proved.
Introduce $H=(H(t),\,t\in\R_+)$,
$H_n=(H_n(t),\,t\in\R_+)$,
$X^{(0)}=(X^{(0)}(t),\,t\in\R_+)$,
$ X_n^{(0)}=( X_n^{(0)}(t),\,t\in\R_+)$,
$U=((U(x,t)\,,x\in\R_+)\,,t\in\R_+)$\,,
$U_n=((U_n(x,t)\,,x\in\R_+)\,,t\in\R_+)$\,,
$\Theta=(\Theta(t)\,,t\in\R_+)$
and $\Theta_n=(\Theta_n(t),\,t\in\R_+)$\,. 
\begin{theorem}
  \label{le:L}
  As $n\to\infty$, the sequence $(X_n(0), X_n^{(0)},H_n,\Theta_n)$
 LD converges in distribution at rate $b_n^2$ in $\R\times\D(\R_+,\R)^3$ to
$(X(0),X^{(0)},H,\Theta)$\,.
\end{theorem}
 The LD convergence of $H_n$ and $X_n^{(0)}$ to their respective idempotent
counterparts, $H$ and $X^{(0)}$\,, is fairly straightforward. 
The proof of the LD convergence of $\Theta_n$ to $\Theta$ is more involved.
The key is   LD convergence of $U_n$ to $U$\,. Since $U_n$ is,
essentially, a time--changed sequential empirical process, one needs a moderate
deviation analogue of the weak convergence of the sequential
empirical process to the Kiefer process, see, e.g., Krichagina and
Puhalskii \cite{kripuh97}. That result is contained in
Lemma \ref{le:kiefer} and may be of interest in its own right.
 Another  important element of the proof
is super--exponential convergence in probability of $\hat A_n(t)/n$ to
$\mu t$\,, which is the subject of Lemma \ref{le:lln}.
The process $\Theta_n$ is then  written as 
\begin{equation}
    \label{eq:41}
\Theta_n(t)=J_n(t)-M_n(t)\,,
\end{equation}
where
\begin{align}
        \label{eq:17}
J_n(t)&=    \int_0^t \frac{U_n(t-x,x)}{1-F(x)}\,dF(x)
\end{align}
and $M_n=(M_n(t)\,, t\in\R_+)$ is a martingale with respect to a
suitable filtration.
By the continuous mapping principle, with  $U_n$ LD converging to $U$\,, 
  $J_n$ LD converges to
$J=(J(t)\,, t\in\R_+)$\,, where
\begin{equation}
  \label{eq:20}
J(t)=\int_0^t \frac{U(t-x,x)}{1-F(x)}\,dF(x)\,.
\end{equation}
It is  proved also that the sequence  $M_n$ LD converges to
a certain idempotent process
$M$\,.
All the convergences holding jointly, 
the result follows.

The plan is being implemented next.
Let random variables  $\zeta_i$ be independent and uniform on $[0,1]$\,.
Define a centred and normalised sequential empirical process by
\begin{align}
      \label{Un}
  K_n(x,t)=\frac{1}{b_n\sqrt{n}}\,\sum_{i=1}^{\lfloor nt \rfloor}
\bl(\ind{\zeta_i\le x}-x\br)\intertext{ and let}
\label{uon}
B_n(x,t)=\frac{1}{b_n\sqrt n}\sum_{i=1}^{\lfloor nt\rfloor}
\bl(\ind{\zeta_i\le x}-\int_0^{x\wedge \zeta_i}\frac{dy}{1-y}\br)\,,
\end{align}
where $x\in[0,1]$\,.
It is a simple matter to check that
\beq{sm}
K_n(x,t)=-\int_0^x
\frac{K_n(y,t)}{1-y}\,dy+B_n(x,t)\,,\;\;t\in\R_+\,,x\in[0,1]\,.
\eeq
Let $K_n=\bl((K_n(x,t),\,x\in[0,1]),t\in\R_+\br)$
and $B_n=\bl((B_n(x,t),\,x\in[0,1]),t\in\R_+\br)$\,.
Let  $K=((K(x,t)\,,x\in[0,1])\,,t\in\R_+)$ represent a Kiefer idempotent
 process
and let  $B=((B(x,t)\,,x\in[0,1])\,,t\in\R_+)$ represent a Brownian sheet
idempotent process, both being canonical coordinate processes on 
$\D(\R_+,\D([0,1],\R^2))$ endowed with deviability $\Pi$ such that
\begin{equation}
  \label{eq:81}
  K(x,t)=-\int_0^x
\frac{K(y,t)}{1-y}\,dy+B(x,t)\,,\;\;t\in\R_+\,,x\in[0,1]\,.
\end{equation}
(That  the solution of \eqref{eq:81} is a Kiefer idempotent process
follows from Lemma \ref{le:idemp_prop}.)

\begin{lemma}
  \label{le:kiefer}
If the convergence in \eqref{eq:3} holds, then 
the sequence of  processes $(K_n,B_n)$
  LD converges at rate $b_n^2$ 
 in $\D(\R_+,\D([0,1],\R^2))$ to the idempotent processes $(K,B)$\,. 
\end{lemma}
\begin{proof}
  The proof is modelled on the proof of Lemma 3.1 in Krichagina and 
Puhalskii \cite{kripuh97}, see also 
 Jacod and Shiryaev \cite[Chapter IX, \S 4c]{jacshir}.
Given $t\in\R_+$\,, 
let $K_n(t)=(K_n(x,t),\,x\in[0,1])\in\D([0,1],\R_+)$\,,
$B_n(t)=(B_n(x,t),\,x\in[0,1])\in\D([0,1],\R_+)$\,,
$K(t)=(K(x,t),\,x\in[0,1])\in\D([0,1],\R_+)$ 
and $B(t)=(B(x,t),\,x\in[0,1])\in\D([0,1],\R_+)\,.$ 
It is proved, first, that, for $0\le t_1<t_2<\ldots<t_k$,
the sequence
$(K_n(t_1),B_n(t_1)),(K_n(t_2),B_n(t_2)),\ldots,(K_n(t_k),B_n(t_k))$ 
LD converges 
to
$(K(t_1),B(t_1)),(K(t_2),B(t_2)),\ldots,(K(t_k),B(t_k)))$ 
 in $\D([0,1],\R^2)^{k}$\,, as $n\to\infty$\,.
Since both the stochastic processes
$((K_n(t),B_n(t))\,,t\in\R_+)$ and the idempotent
process $((K(t),B(t))\,,t\in\R_+)$ have independent increments in $t$\,,
it suffices to prove  convergence of one--dimensional
distributions, so,  $((K_n(x,t),B_n(x,t))\,,x\in[0,1])$
and $((K(x,t),B(x,t))\,,x\in[0,1])$ are worked with, with $t$ being held fixed.
The process $(B_n(x,t)\,,x\in[0,1])$
is a  martingale  with respect to the natural filtration
with the jump measure
\[
  \mu^{n,B}([0,x],\Gamma)=\ind{1/(b_n\sqrt{n})\in\Gamma}
\sum_{i=1}^{\lfloor nt\rfloor}
\ind{\zeta_i\le x}\,,
\]
 the predictable jump measure
\[
  \nu^{n,B}([0,x],\Gamma)=\ind{1/(b_n\sqrt{n})\in\Gamma}
\sum_{i=1}^{\lfloor nt\rfloor}
\int_0^{x\wedge \zeta_i}\frac{dy}{1-y}
\]
and  the predictable quadratic variation process 
\begin{multline}
  \label{eq:53}
    \langle B_n\rangle (x,t)=\int_0^x\int_\R u^2\nu^{n,B}(dy,du)=
\frac{1}{b_n^2n}\,\nu^{n,B}([0,x],\{1/(b_n\sqrt{n})\})=
\frac{1}{b_n^2n}\,\sum_{i=1}^{\lfloor nt\rfloor}
\int_0^{x\wedge \zeta_i}\frac{dy}{1-y}\\=
\frac{\lfloor nt\rfloor}{b_n^2n}\,x+
\frac{1}{b_n\sqrt{n}}\,K_n(x,t)-\frac{1}{b_n\sqrt{n}}\,B_n(x,t)\,,
\end{multline}
where $\Gamma\subset\R\setminus\{0\}$\,.
It is shown next that
\begin{equation}
  \label{eq:48}
  \lim_{r\to\infty}\limsup_{n\to\infty}P(
\sup_{x\in[0,1]}\abs{B_n(x,t)}>r)^{1/b_n^2}=0\,.
\end{equation}
An exponential martingale device is applied, which will be used
repeatedly.
Since,  the process 
\begin{multline*}
   \exp\bl( b_n^2 B_n(x,t)-
\int_0^x\int_\R(
e^{ b_n^2 u}
-1- b_n^2u)\nu^{n,B}(dy,du)\br)\\
=   \exp\bl( b_n^2 B_n(x,t)-(
e^{
  b_n/\sqrt{n}}
-1- \frac{b_n}{\sqrt{n}})\sum_{i=1}^{\lfloor nt\rfloor}
\int_0^{x\wedge \zeta_i}\frac{dy}{1-y}\br)
\end{multline*}
is a local martingale with respect to $x$\,, see, e.g., Lemma 4.1.1 on p.294 in
Puhalskii \cite{Puh01},  for any stopping time $\tau$\,,
\begin{equation}
  \label{eq:54}
    E\exp\bl( b_n^2 B_n(\tau,t)-(
e^{
  b_n/\sqrt{n}}
-1- \frac{b_n}{\sqrt{n}})\sum_{i=1}^{\lfloor n t\rfloor}
\int_0^{\tau\wedge \zeta_i}\frac{dy}{1-y}\br)\le 1\,.
\end{equation}
Lemma 3.2.6 on p.282 in Puhalskii \cite{Puh01} implies that, for
 $r>0$ and $\gamma>0$\,, 
\begin{multline*}
    P(\sup_{x\in[0,1]}e^{ b_n^2  B_n(x,t)}\ge e^{
  b_n^2 r})
\le e^{ b_n^2(\gamma- r)}
+P\bl(\exp\bl((e^{
  b_n/\sqrt{n}}
-1- \frac{b_n}{\sqrt{n}})\sum_{i=1}^{\lfloor n t\rfloor}
\int_0^{ \zeta_i}\frac{dy}{1-y}\br)\ge e^{
  b_n^2\gamma}\br)
\\\le
e^{ b_n^2(\gamma- r)}
+ e^{-
  b_n^2\gamma}E\bl(\exp\bl((e^{
  b_n/\sqrt{n}}
-1- \frac{b_n}{\sqrt{n}})\sum_{i=1}^{\lfloor n t\rfloor}
\int_0^{ \zeta_i}\frac{dy}{1-y}\br)\br)\\
\\=e^{ b_n^2(\gamma- r)}
+ e^{-
  b_n^2\gamma}\bl(1-(e^{
  b_n/\sqrt{n}}
-1- b_n/\sqrt{n})\br)^{-\lfloor nt\rfloor} \,,
\end{multline*}
with the latter equality holding for all $n$ great enough because
$e^{
  b_n/\sqrt{n}}
-1- b_n/\sqrt{n}\to0$\,.
Hence,  assuming that $e^{
  b_n/\sqrt{n}}
-1- b_n/\sqrt{n}\le1/2$\,,
\[
  P(\sup_{x\in[0,1]}e^{ b_n^2  B_n(x,t)}\ge e^{
  b_n^2 r})^{1/b_n^2}
\le e^{ (\gamma- r)}
+e^{-
  \gamma}
2^{\lfloor nt\rfloor/b_n^2}\,.
\]
Since $n/b_n^2\to\infty$\,, it follows that
\begin{equation}
  \label{eq:49}
  \lim_{ r\to\infty}\limsup_{n\to\infty}  P(\sup_{x\in[0,1]}  B_n(x,t)\ge 
   r)^{1/b_n^2}=0\,.
\end{equation}
A similar convergence
holds with $-B_n(x,t)$ substituted for $B_n(x,t)$\,. 
The limit in \eqref{eq:48} has been proved.

It is proved next that
\begin{equation}
  \label{eq:43}
  \lim_{r\to\infty}\limsup_{n\to\infty}P(
\sup_{x\in[0,1]}\abs{K_n(x,t)}>r)^{1/b_n^2}=0\,.
\end{equation}
Since owing to \eqref{Un}, $(K_n(x,t),x\in[0,1])$ is distributed as 
$(-K_n(1-x,t),
x\in[0,1])$ \,, 
 it suffices to prove that
\begin{align}
  \label{eq:50}
\lim_{r\to\infty}\limsup_{n\to\infty}
P(
\sup_{x\in[0,1/2]}\abs{K_n(x,t)}>r)^{1/b_n^2}=0\,.
\end{align}
By \eqref{sm} and the Gronwall--Bellman inequality,
\begin{equation}
  \label{eq:51}
  \sup_{x\in[0,1/2]}\abs{K_n(x,t)}\le e^2\sup_{x\in[0,1/2]}\abs{B_n(x,t)}
\end{equation}
so that \eqref{eq:43} holds too.

By \eqref{eq:53}, \eqref{eq:43} and \eqref{eq:48}, for $x\in[0,1]$\,,
\begin{equation}
  \label{eq:140}
    \lim_{n\to\infty}P(\abs{b_n^2\langle B_n\rangle
    (x,t)-tx}>\epsilon)^{1/b_n^2}=0\,. 
\end{equation}
If extended past $x=1$ by letting $B_n(x,t)=B_n(1,t)$\,, 
the process $(B_n(x,t)\,,x\in\R_+)$ is a square integrable
martingale with predictable quadratic variation process
$(\langle B_n\rangle(x\wedge1,t)\,, x\in\R_+)$\,, so,
 by \eqref{eq:140} and
Theorem 5.4.4 on p.423 in Puhalskii \cite{Puh01}, where one takes
$\beta_\phi=b_n\sqrt{n}$\,, $\alpha_\phi=n$ and $r_\phi=b_n^2$\,, the 
sequence of the extended processes $B_n(t)$ LD converges in $\D(\R_+,\R)$
 to the idempotent process
$(B(x\wedge1,t)\,,x\in[0,1])$\,.
By \eqref{sm} and the continuous mapping principle,
for $0\le x_1\le\ldots\le x_l<1$\,, the
$((K_n(x_i,t),B_n(x_i,t))\,,i\in\{0,1,2,\ldots,l\})$ LD 
converge in $\D(\R_+,\R^l)$ to 
$((K(x_i,t),B(x_i,t))\,,i\in\{0,1,2,\ldots,l\})$\,.
Since $K_n(1,t)=0$ and $K(1,t)=0$ (see Appendix \ref{sec:large-devi-conv})\,, 
the latter convergence also holds if $x_l=1$\,.

It is shown now that the sequence 
$(K_n(x,t)\,,x\in[0,1])$ is $\bbC$--exponentially
tight of order $b_n^2$ in $\D([0,1],\R)$\,. (The definition and
basic properties of $\bbC$--exponential tightness are reviewed in 
Appendix \ref{sec:large-devi-conv}.)
 By  Theorem \ref{the:exp_tight}, the convergence in \eqref{eq:43}
needs to be complemented with
\begin{equation}
  \label{eq:127a}
  \lim_{\delta\ra0}\limsup_{n\to\infty}\sup_{x\in[0,1]}
P(\sup_{y\in[0, \delta]}|K_n(x+y,t)-K_n(x,t)|\ge \eta)^{1/b_n^2}=0\,,
\end{equation}
for arbitrary $\eta>0$\,, where $K_n(x,t)=0$ when $x\ge1$\,.
A similar argument is used as in the proof of \eqref{eq:43}. Defining
$\overline{K}_n(x,t)=-K_n(1-x,t)$\,, for $x\in[0,1]$\,, and 
$\overline{K}_n(x,t)=0$ for $x\ge 1$\,,
implies, by \nr{Un}, that
\begin{multline}
  \label{eq:129}
  \sup_{x\in[0,1]}
P(\sup_{y\in[0, \delta]}|K_n(x+y,t)-K_n(x,t)|\ge \eta)\le
 \sup_{x\in[0,1/2]}
P(\sup_{y\in[0, \delta]}|K_n( x+y,t)-K_n(x,t)|\ge \eta)\\+
\sup_{x\in[1/2,1-\delta]}
P(\sup_{y\in[0, \delta]}|K_n(x+y,t)- K_n(x,t)|\ge \eta)
+\sup_{x\in[1-\delta,1]}
P(\sup_{y\in[0, \delta]}|K_n(x+y,t)- K_n(x,t)|\ge \eta)\\
\le\sup_{x\in[0,1/2]}
P(\sup_{y\in[0, \delta]}|K_n( x+y,t)-K_n(x,t)|\ge \eta)+
\sup_{x\in[\delta,1/2]}
P(\sup_{y\in[0, \delta]}|K_n(1-x+y,t)- K_n(1-x,t)|\ge \eta)\\
+
P(\sup_{u\in[1- \delta,1]}|K_n(u,t)|\ge \eta/2)
\le\sup_{x\in[0,1/2]}
P(\sup_{y\in[0, \delta]}|K_n( x+y,t)-K_n(x,t)|\ge \eta)\\+
\sup_{x\in[\delta,1/2]}
P(\sup_{y\in[0, \delta]}|\overline K_n(x-y,t)-\overline K_n(x,t)|\ge \eta)
+
P(\sup_{u\in[0, \delta]}|\overline K_n(u,t)|\ge \eta/2)\\
\le\sup_{x\in[0,1/2]}
P(\sup_{y\in[0, \delta]}|K_n( x+y,t)-K_n(x,t)|\ge \eta)+
\sup_{x\in[0,1/2]}
P(\sup_{y\in[0, \delta]}|\overline K_n(x+y,t)-\overline K_n(x,t)|\ge \eta)\\
+
P(\sup_{u\in[0, \delta]}|\overline K_n(u,t)|\ge \eta/2)\,.
\end{multline}
Since the random variables $\zeta_i$ are independent and  uniformly distributed
on $[0,1]$\,, $\overline{K}_n$ has the same
finite dimensional distributions as  $K_n$ so that
\begin{multline}
  \label{eq:126'}
    \sup_{x\in[0,1]}
P(\sup_{y\in[0, \delta]}|K_n(x+y,t)-K_n(x,t)|\ge \eta)\le
2\sup_{x\in[0,1/2]}
P(\sup_{y\in[0, \delta]}|K_n( x+y,t)-K_n(x,t)|\ge \eta)\\
+
P(\sup_{u\in[0, \delta]}| K_n(u,t)|\ge \eta/2)\,.
\end{multline}
Since $x+y\le 2/3$\,, when $x\in[0,1/2]$ and $y\in[0,\delta]$\,,
provided $\delta$ is small enough,
by \eqref{sm}, for $x\in[0,1/2]$ and $\delta$ small enough,
\begin{equation}
  \label{eq:130'}
\sup_{y\in[0,\delta]}\abs{  K_n(x+y,t)-K_n(x,t)}\le3\delta
\sup_{u\in[0,1]}\abs{K_n(u,t)}+\sup_{y\in[0,\delta]}\abs{B_n(x+y,t)-B_n(x,t)}\,.
\end{equation}
Similarly,
\begin{equation}
  \label{eq:132'}
  \sup_{x\in[0, \delta]}| K_n(x,t)|\le\frac{\delta}{1-\delta}\,
\sup_{x\in[0,1]}\abs{K_n(x,t)}+\sup_{x\in[0, \delta]}| B_n(x,t)|\,.
\end{equation}
By
 \eqref{eq:126'}, \eqref{eq:130'},
\eqref{eq:132'} and the fact that the $B_n$ LD converge to $B$\,,
\begin{multline}
  \label{eq:133}
\limsup_{n\to\infty}\sup_{x\in[0,1]}
P(\sup_{y\in[0, \delta]}|K_n(x+y,t)-K_n(x,t)|\ge \eta)^{1/b_n^2}  
\\
\le\limsup_{n\to\infty}P(3\delta
\sup_{u\in[0,1]}\abs{K_n(u,t)}\ge\eta/2)^{1/b_n^2}
+\limsup_{n\to\infty}P(\frac{\delta}{1-\delta}\,
\sup_{u\in[0,1]}\abs{K_n(u,t)}\ge\eta/4)^{1/b_n^2}\\
+\Pi(\sup_{x\in[0,1/2]}\sup_{y\in[0,\delta]}\abs{B(x+y,t)-B(x,t)}\ge \eta/2)+
\Pi(\sup_{u\in[0, \delta]}| B(u,t)|\ge\eta/4)\,.
\end{multline}
By \eqref{eq:43},
\begin{multline}
  \label{eq:135'}
  \lim_{\delta\to0}\limsup_{n\to\infty}\sup_{x\in[0,1]}
P(\sup_{y\in[0, \delta]}|K_n(x+y,t)-K_n(x,t)|\ge \eta)^{1/b_n^2}  \\
\le\lim_{\delta\to0} \Pi(\sup_{x\in[0,1/2]}\sup_{y\in[0,\delta]}\abs{B(x+y,t)-B(x,t)}\ge \eta/2)+
\lim_{\delta\to0}\Pi(\sup_{u\in[0, \delta]}| B(u,t)|\ge\eta/4)\,.
\end{multline}
Since the collections of sets
$\{b\in\bbC(\R_+,\bbC([0,1],\R)) :\,\sup_{x\in[0,1/2]}\sup_{y\in[0,\delta]}\abs{b(x+y,t)-b(x,t)}\ge
\eta/2\}$ and $\{b\in\bbC(\R_+,\bbC([0,1],\R)) :
\,\sup_{x\in[0,\delta]}\abs{b(x,t)}\ge \eta/4\}
$ are nested collections of closed sets as $\delta\downarrow0$\,, the limit on the right of
\eqref{eq:135'}, is (see Appendix \ref{sec:large-devi-conv})
\[
  \Pi(\sup_{x\in[0,1/2]}\sup_{y\in[0,0]}\abs{B(x+y,t)-B(x,t)}\ge \eta/2)+
\Pi(\sup_{u\in[0,0]}| B(u,t)|\ge\eta/4)=0\,,
\]
which concludes the proof of \eqref{eq:127a}.

Since the sequence $((K_n(x,t),B_n(x,t))\,,x\in[0,1])$ LD converges to
$((K(x,t),B(x,t))\,,x\in[0,1])$ in the sense of finite--dimensional
distributions and is $\bbC$--exponentially tight, the LD convergence
holds in $\D([0,1],\R^2)$\,, see Theorem \ref{the:fd}.
It has thus been proved that the sequence 
$\bl(((K_n(x,t_1),B_n(x,t_1)),\,x\in[0,1]),\ldots,
((K_n(x,t_l),B_n(x,t_l)),\,x\in[0,1])\br)$ LD converges
in $\D([0,1],\R^2)^l$  to
$\bl(((K(x,t_1),B(x,t_1)),\,x\in[0,1]),\ldots,
((K(x,t_l),B(x,t_l)),\,x\in[0,1])\br)$\,, for all $t_1\le t_2\le\ldots\le t_l$\,.
The proof of the lemma will be complete if the sequence
$\bl(((K_n(x,t),B_n(x,t)),x\in[0,1])\,,t\in\R_+\br)$ is
proved to be 
$\bbC$--exponentially tight of order $b_n^2$ in $\D(\R_+,\D([0,1],\R^2))$\,.
The definition of being exponentially tight implies that
it is sufficient to prove that each of the sequences
$\{K_n,n\ge1\}$
and $\{B_n,n\ge1\}$ is $\bbC$--exponentially tight 
of order $b_n^2$ in $\D(\R_+,\D([0,1],\R))$. 
By \eqref{eq:43} and
Theorem \ref{the:exp_tight}, the $\bbC$--exponential tightness of
$\{K_n,n\ge1\}$ would follow if,
for all $L>0$ and $\eta>0$\,,
\beray{T2}
\lim_{\delta\ra0}\limsup_{n\to\infty}\sup_{s\in[0,L]}
P(\sup_{t\in[0, \delta]}\sup_{x\in[0,1]}|K_n(x,s+t)-K_n(x,s)|\ge \eta)^{1/b_n^2}=0.
\end{eqnarray}
Since, in analogy with the reasoning in \eqref{eq:129},
\begin{multline*}
  P(\sup_{t\in[0, \delta]}\sup_{x\in[0,1]}|K_n(x,t+s)-K_n(x,s)|\ge \eta)\le
P(\sup_{t\in[0, \delta]}\sup_{x\in[0,1/2]}|K_n(x,t+s)-K_n(x,s)|\ge \eta)\\+
P(\sup_{t\in[0, \delta]}\sup_{x\in[1/2,1]}|K_n(x,t+s)-K_n(x,s)|\ge \eta)=
P(\sup_{t\in[0, \delta]}\sup_{x\in[0,1/2]}|K_n(x,t+s)-K_n(x,s)|\ge \eta)\\+
P(\sup_{t\in[0, \delta]}\sup_{x\in[0,1/2]}
|\overline{K}_n(x,t+s)-\overline{K}_n(x,s)|\ge \eta)=
2P(\sup_{t\in[0, \delta]}\sup_{x\in[0,1/2]}|K_n(x,t+s)-K_n(x,s)|\ge \eta)\,,
\end{multline*}
 \nr{T2} is implied by 
\beray{T2'2}
\lim_{\delta\ra0}\limsup_n\sup_{s\in[0,L]}
P(\sup_{t\in[0, \delta]}\sup_{x\in[0,1/2]}|K_n(x,t+s)-K_n(x,s)|\ge \eta)^{1/b_n^2}=0.
\end{eqnarray}

With $x$ being held fixed, 
the process $(K_n(x,t+s)-K_n(x,s),t\in\R_+)$ is a locally square--integrable
martingale,
so $(\sup_{x\in[0,1/2]}
(K_n(x,t+s)-K_n(x,s)),t\in\R_+)$, is a submartingale, 
hence by Doob's inequality, Liptser and Shiryaev 
\cite[Theorem 3.2, p.60]{LipShi01},
\beray{T21}
P(\sup_{t\in[0, \delta]}\sup_{x\in[0,1/2]}|K_n(x,t+s)-K_n(x,s)|\ge \eta)
\le\frac{1}{\eta^{2b_n^2}}\,
E\sup_{x\in[0,1/2]}(K_n(x,s+\delta)-K_n(x,s))^{2b_n^2}.
\end{eqnarray}
As noted earlier, with $t$ being held fixed, the process
 $(B_n(x,t),x\in[0,1])$ is a square--integrable
martingale,
Jacod and Shiryaev \cite[Chapter II, \S 3c]{jacshir}. Equation \nr{sm}
yields, by the Gronwall--Bellman  inequality, as in \eqref{eq:51},
\beray{e2}
\sup_{x\in[0,1/2]}|K_n(x,s+\delta)-K_n(x,s)|\le
e^2\sup_{x\in[0,1]}|B_n(x,s+\delta)-B_n(x,s)|\,.
\end{eqnarray}
Since  $(B_n(x,s+\delta)-B_n(x,s),
x\in[0,1])$ is a square--integrable martingale,
by another application of Doob's inequality,
Jacod and Shiryaev \cite[Theorem I.1.43]{jacshir}, 
Liptser and Shiryaev \cite[Theorem I.9.2]{lipshir}, as well as
by Jensen's inequality,
\beray{d2}
E\sup_{x\in[0,1]}(B_n(x,s+\delta)-B_n(x,s))^{2b_n^2}
\le \Bl(\frac{2b_n^2}{2b_n^2-1}\Br)^{2b_n^2}
E(B_n(1,s+\delta)-B_n(1,s))^{2b_n^2}\,.
\end{eqnarray}
By \nr{uon} and the bound  (5.6) in the proof of Theorem 19, Chapter
III, \S5 in Petrov \cite{Pet87},
\begin{multline}
  \label{eq:98}
  E(B_n(1,s+\delta)-B_n(1,s))^{2b_n^2}\le
(b_n\sqrt{n})^{-2b_n^2}\bl((b_n^2+1)^{2b_n^2} n\delta E(1+\ln\zeta_1)^{2b_n^2}\\+
2b_n^2(b_n^2+1)^{b_n^2}e^{2b_n^2}
  (n\delta)^{b_n^2}\bl(
E(1+\ln\zeta_1)^2\br)^{b_n^2}\br)
\end{multline}
so that
\begin{multline*}
\limsup_{n\to\infty}(
E(B_n(1,s+\delta)-B_n(1,s))^{2b_n^2})^{1/b_n^2}\le
\limsup_{n\to\infty}
(b_n\sqrt{n})^{-2}\bl((b_n^2+1)^{2}n^{1/b_n^2} (E(1+\ln\zeta_1)^{2b_n^2})^{1/b_n^2}\\+
(b_n^2+1)e^{2}
  n\delta
E(1+\ln\zeta_1)^2\br)
\le \limsup_{n\to\infty}
(b_n\sqrt{n})^{-2}\bl((b_n^2+1)^{2}\bl((2b_n^2)!\br)^{1/b_n^2}+
(b_n^2+1)e^{2}
  n\delta
E(1+\ln\zeta_1)^2\br)
\end{multline*}
which implies, on recalling that 
$b_n^3n^{1/b_n^2-1/2}\to0$\,, that
\[
\lim_{\delta\to0}\limsup_{n\to\infty}
\sup_{s\in[0,L]}\Bl(E(B_n(1,s+\delta)-B_n(1,s))^{2b_n^2}\Br)^{1/b_n^2}=0\,.
\]
Recalling  \nr{T21}, \nr{e2} and \nr{d2} obtains \nr{T2'2}.
The proof of the $\bbC$--exponential tightness of $B_n$ is similar. (It
is actually simpler.)
Finally, \eqref{eq:81} is obtained by letting $n\to\infty$ in \eqref{sm}.

\end{proof}
The proof of Theorem \ref{le:L} is in order now, so,
the hypotheses of Theorem \ref{the:heavy_traffic} are assumed.
Since the $F(\eta_i)$ are  iid and uniform on $[0,1]$\,,
in view of  \eqref{eq:33} and \eqref{Un},
 it may be
assumed that
\begin{equation}
  \label{eq:80}
  U_n(x,t)=K_n(F(x),\frac{\hat A_n(t)}{n})\,.
\end{equation}
On recalling  \eqref{eq:26},
thanks to the large deviation version of Slutsky's theorem, see
Theorem \ref{the:sl}, the LD convergence of
 $U_n$             to $U$
in $\D(\R_+,\D(\R_+,\R))$
 is a consequence  of Lemma \ref{le:kiefer} and the following result. 
\begin{lemma}
\label{le:lln}
For arbitrary $\epsilon>0$ and $t>0$\,,
  \[
\lim_{n\to\infty}P\bl(
\sup_{s\in[0,t]}\abs{\frac{\hat A_n(s)}{n}-\mu s}>\epsilon\br)^{1/b_n^2}=0\,.
\]
\end{lemma}
First, the groundwork is laid and gaps left in the earlier presentation
are filled in.
Let 
\begin{equation}
  \label{eq:P9}
  L_n(x,t)=\frac{1}{b_n\sqrt{n}}\,\sum_{i=1}^{{{\hat A}_n}(t)}
\Bl(\ind{\eta_i\le x}-\int_0^{\eta_i\wedge x}
\frac{dF(u)}{1-F(u)}\Br)\,.
\end{equation}
By  \eqref{uon}, it may be assumed that 
\begin{equation}
  \label{eq:7}
L_n(x,t)=B_n(F(x),\frac{\hat A_n(t)}{n})\,.
\end{equation}
  By (\ref{eq:33}),
\begin{equation}
  \label{eq:P10}
  U_n(x,t)=-\int_0^x\frac{U_n(y,t )}{1-F(y)}\,dF(y)+L_n(x,t)
\end{equation}
so that, owing to \eqref{eq:19}, 
 \eqref{eq:41} holds with $J_n(t)$ defined by \eqref{eq:17} and with
\begin{align}
  \label{eq:P18}
  M_{n}(t)&=\int_{\R_+^2} \ind{s+x\le t}\,dL_n(x,s)
\,.\end{align}

Two results on what could be called
''asymptotic boundedness in
probability of order $b_n^2$'' are proved next.
More specifically, it is proved that, for $t>0$\,,
\begin{align}
    \label{eq:23}
    \lim_{r\to\infty}\limsup_{n\to\infty}
P(\sup_{s\in[0,t]}\abs{U_n(s,t-s)}>r)^{1/b_n^2}=0
 \\\intertext{and}
\label{eq:23a}    \lim_{r\to\infty}\limsup_{n\to\infty}
P(\sup_{s\in[0,t]}\abs{M_n(s)}>r)^{1/b_n^2}=0\,.
  \end{align}
  \begin{proof}[Proof of \eqref{eq:23}]
    Note that
\begin{equation}
    \label{eq:5}
\lim_{r\to\infty}    \limsup_{n\to\infty}P\bl(\frac{A_n(t)}{n}>r\br)^{1/b_n^2}=0\,,
  \end{equation}
which is a consequence of the LD convergence at rate $b_n^2$ of the
$Y_n$ to $Y$ (see \eqref{eq:46}).
Similarly, the LD convergence of  $X_n(0)$ to $X(0)$ implies that
\begin{equation}
  \label{eq:72}
\lim_{r\to\infty}  \limsup_{n\to\infty}
P\bl(\sup_{s\in[0,t]}\frac{Q_n(0)}{n}>r\br)^{1/b_n^2}=0\,.
\end{equation}
By \eqref{eq:40},
${{\hat A}_n}(t)\le(Q_n(0)-n)^++A_n(t)$
so that  \eqref{eq:5} and \eqref{eq:72} imply that
 \begin{equation}
   \label{eq:29}
\lim_{L\to\infty}       \limsup_{n\to\infty}P\bl(
\frac{\hat A_n(t)}{n}>L\br)^{1/b_n^2}=0\,.
 \end{equation}
By \eqref{eq:80},
\begin{equation}
  \label{eq:44}
  P(\sup_{s\in[0,t]}\abs{U_n(s,t-s)}>r)\le
P(\frac{\hat A_n(t)}{n}>L)+
P(\sup_{
\substack{s\in [0,L],\\x\in[0,1]}}\abs{K_n(x,s)}>r)\,.
\end{equation}
By the LD convergence of  $K_n$ to $K$  in Lemma
\ref{le:kiefer} and the trajectories of $K$ being continuous,
\[
  \lim_{r\to\infty}       \limsup_{n\to\infty}P\bl(
\sup_{
\substack{s\in [0,L],\\x\in[0,1]}}\abs{K_n(x,s)}>r)^{1/b_n^2}=0\,.
\]
Combined with \eqref{eq:29} and \eqref{eq:44}, this proves \eqref{eq:23}.
  \end{proof}

  \begin{proof}[Proof of \eqref{eq:23a}]
As follows from  Lemma 3.1 in Puhalskii and Reed \cite{PuhRee10},
the process $M_n=(M_n(t),\,t\in\R_+)$ 
 is a local martingale with respect to  filtration $\mathbf{G}_n$
 defined as follows.
For $t\in\R_+$, let 
$\hat{\mathcal{G}}_n(t)$ denote the complete  $\sigma$-algebra generated by
the random variables
$\ind{\hat\tau_{n,i}\le s}\ind{ \eta_i\le x}$, where $x+ s\le t$ and
$i\in\N$, and by the
${{\hat A}_n}(s)$ (or, equivalently, by the $\ind{\hat\tau_{n,i}\le s}$ for
$i\in\N$), where $s\leq t$.   Define
$\mathcal{G}_n(t)=\cap_{\epsilon>0}\hat{\mathcal{G}}_n(t+\epsilon)$
and $\mathbf{G}_n=(\mathcal{G}_n(t)\,, t\in\R_+)$\,.
By \eqref{eq:P9} and \eqref{eq:P18},
\begin{equation}
  \label{eq:9}
  M_n(t)=\frac{1}{b_n\sqrt{n}}\,\sum_{i=1}^{{{\hat A}_n}(t)}
\bl(\ind{\eta_i+\hat\tau_{n,i}\le t}-\int_0^{\eta_i\wedge (t-\hat\tau_{n,i})}
\frac{dF(u)}{1-F(u)}\br)\,.
\end{equation}
Thus, the measure of jumps of $M_n$ is
\begin{equation}
  \label{eq:18}
  \mu_n([0,t],\Gamma)=\ind{1/(b_n\sqrt{n})\in\Gamma}
\sum_{i=1}^{{{\hat A}_n}(t)}
\ind{\eta_i+\hat\tau_{n,i}\le t}
\end{equation}
and the associated $\mathbf{G}_n$--predictable measure of jumps is
\begin{equation}
  \label{eq:25}
  \nu_n([0,t],\Gamma)=\ind{1/(b_n\sqrt{n})\in\Gamma}
\sum_{i=1}^{{{\hat A}_n}(t)}\int_0^{\eta_i\wedge (t-\hat\tau_{n,i})}
\frac{dF(u)}{1-F(u)}\,.
\end{equation}
Note that it is a continuous process. (For $\hat A_n$ being
predictable, see Lemma C.1 in Puhalskii and Reed \cite{PuhRee10}.)
The associated stochastic cumulant  is (see, e.g., p.293 in 
Puhalskii \cite{Puh01})
\begin{equation}
  \label{eq:27}
  G_n(\alpha,t)=(e^{\alpha/(b_n\sqrt{n})}-1-\frac{\alpha}{b_n\sqrt{n}})
\sum_{i=1}^{{{\hat A}_n}(t)}\int_0^{\eta_i\wedge (t-\hat\tau_{n,i})}
\frac{dF(u)}{1-F(u)}\,.
\end{equation}
By Lemma 4.1.1 on p.294 in Puhalskii \cite{Puh01}, 
the process $(e^{\alpha
  M_n(t)-G_n(\alpha,t)}\,,t\in\R_+)$ is a local martingale so that
$E e^{\alpha
  M_n(\tau)-G_n(\alpha,\tau)}\le1$\,, for arbitrary stopping time
$\tau$\,.
Lemma 3.2.6 on p.282 in Puhalskii \cite{Puh01} implies that, for
  $\gamma>0$\,, 
\[
    P(\sup_{s\in[0,t]}e^{\alpha b_n^2  M_n(s)}\ge e^{\alpha
  b_n^2 r})
\le e^{\alpha b_n^2(\gamma- r)}
+P(e^{G_n(\alpha b_n^2,t)}\ge e^{\alpha
  b_n^2\gamma})
\le  e^{\alpha b_n^2(\gamma- r)}
+P(e^{\hat G_n(\alpha b_n^2,t)}\ge e^{\alpha
  b_n^2\gamma})\,,
\]
where 
\begin{equation}
  \label{eq:6}
    \hat G_n(\alpha,t)=(e^{\alpha/(b_n\sqrt{n})}-1-\frac{\alpha}{b_n\sqrt{n}})
\sum_{i=1}^{{{\hat A}_n}(t)}\int_0^{t-\hat\tau_{n,i}}
\frac{dF(u)}{1-F(u)}\,.
\end{equation}
Hence, for $\alpha>0$\,, 
\begin{equation}
  \label{eq:56}
  P(\sup_{s\in[0,t]} M_n(s)\ge r)^{1/b_n^2}\le
e^{\alpha(\gamma- r)}
+P(\hat G_n(\alpha b_n^2,t)\ge\alpha
  b_n^2\gamma)^{1/b_n^2}\,.
\end{equation}
On writing
\begin{multline}
  \label{eq:57}
      \hat G_n(\alpha b_n^2,t)=
(e^{\alpha b_n/\sqrt{n}}-1-\frac{\alpha b_n}{\sqrt{n}})
\int_0^t\int_0^{t-s}
\frac{dF(u)}{1-F(u)}\,d\hat A_n(s)\\
=(e^{\alpha b_n/\sqrt{n}}-1-\frac{\alpha b_n}{\sqrt{n}})
\int_0^t\hat A_n(t-u)\,\frac{dF(u)}{1-F(u)}
\end{multline}
and noting that $(n/b_n^2) (e^{\alpha b_n/\sqrt{n}}-1-\alpha
  b_n/\sqrt{n})\to
\alpha^2/2$\,,
one can see, thanks to \eqref{eq:29}, that
 \[
   \lim_{n\to\infty}P(\hat G_n(\alpha b_n^2,t)\ge\alpha
  b_n^2\gamma)^{1/b_n^2}=0\,,
 \]
 provided $\alpha$ is
small enough, which proves that
\[
    \lim_{r\to\infty}
\limsup_{n\to\infty}P(\sup_{s\in[0,t]}M_n(s)>r)^{1/b_n^2}=0\,.
\]
The argument for $\sup_{s\in[0,t]}(-M_n(s))$ is similar.
  \end{proof}

\begin{proof}[Proof of Lemma \ref{le:lln}]
  
By \eqref{eq:40}, (\ref{eq:107}), \eqref{eq:33} and \eqref{eq:19},
  \begin{multline}
    \label{eq:P6}
\frac{1}{n}\,Q_n(t)
=\bl(\frac{1}{n}\,Q_n(0)-1\br)^+(1-F(t))
+\frac{1}{n}\,Q_n^{(0)}(t)
+\frac{1}{n}\,A_n(t)-\frac{1}{n}\,
\int_{0}^tA_n(t-s)\,dF(s)\\
+\frac{1}{n}\,\int_0^t
(Q_n(t-s)-n)^+ \,dF(s)
+\frac{b_n}{\sqrt{n}}\,\Theta_n(t)\,.
\end{multline}
  By \eqref{eq:41}, \eqref{eq:17}, \eqref{eq:23} and \eqref{eq:23a},
 on recalling that 
$F(t)<1$\,, 
\begin{equation}
  \label{eq:13}
  \lim_{r\to\infty}\limsup_{n\to\infty}P(\sup_{s\in[0,t]}\abs{\Theta_n(s)}>r)^{1/b_n^2}=0\,.
\end{equation}
The LD convergence at rate $b_n^2$ of 
$Y_n$ to $Y$
implies that, for $\epsilon>0$\,,
\begin{equation}
  \label{eq:4}
      \lim_{n\to\infty}P\bl(\sup_{s\in[0,t]}\abs{\frac{A_n(s)}{n}-\mu
        s}>\epsilon\br)^{1/b_n^2}=0\,.
\end{equation}

By  \eqref{eq:1a}, \eqref{eq:72} and Lemma \ref{le:kiefer},
\begin{equation}
  \label{eq:74}
  \lim_{n\to\infty}P(\sup_{s\in[0,t]}
\abs{\frac{1}{n}\,Q_n^{(0)}(s)-(1-F_0(s))}>\epsilon)^{1/b_n^2}=0\,.
\end{equation}
Recalling  \eqref{eq:78} and \eqref{eq:5} 
implies that
\begin{equation}
  \label{eq:75}
  \lim_{n\to\infty}P(\sup_{s\in[0,t]}\abs{\frac{1}{n}\,Q_n^{(0)}(s)
+\frac{1}{n}\,A_n(s)-\frac{1}{n}\,
\int_{0}^sA_n(s-x)\,dF(x)-1}>\epsilon)^{1/b_n^2}=0\,.
\end{equation}
Besides, the LD convergence of  $X_n(0)$ to $X(0)$ implies that
\eqref{eq:72} can be strengthened as follows:
\begin{equation}
  \label{eq:72'}
  \limsup_{n\to\infty}
P\bl(\abs{\frac{Q_n(0)}{n}-1}>\epsilon\br)^{1/b_n^2}=0\,.
\end{equation}
Hence, by  \eqref{eq:P6} and \eqref{eq:13},
\begin{equation}
  \label{eq:76}
  \frac{1}{n}\,Q_n(t)-1=\,\int_0^t
(\frac{1}{n}\,Q_n(t-s)-1)^+ \,dF(s)+\theta_n(t)\,,
\end{equation}
where 
\[
\lim_{n\to\infty}  P(\sup_{s\in[0,t]}\abs{\theta_n(s)}>\epsilon)^{1/b_n^2}=0\,.
\]
Lemma B.1 in Puhalskii and Reed \cite{PuhRee10} implies that there
exists function $\rho$ which depends on the function $F$ only
such that
\[
  \sup_{s\in[0,t]}\abs{\frac{1}{n}\,Q_n(s)-1}\le \rho(t)
\sup_{s\in[0,t]}\abs{\theta_n(s)}\,.
\]
Therefore,
\begin{equation}
  \label{eq:77}
  \lim_{n\to\infty}  P(\sup_{s\in[0,t]}
\abs{\frac{1}{n}\,Q_n(s)-1}>\epsilon)^{1/b_n^2}=0\,.
\end{equation}
When combined with
 \eqref{eq:40} and \eqref{eq:4}, this yields the assertion of the lemma.
\end{proof}

As mentioned,   the LD convergence of  $U_n$ to $U$  has
been proved thereby.
Let $L_n=\bl((L_n(x,t)\,,x\in\R_+),t\in\R_+\br)$
 and
$L=\bl((L(x,t)\,,x\in\R_+),t\in\R_+\br)$\,, with
$L(x,t)=B(F(x),\mu t)$\,. 
\begin{lemma}
  \label{le:sovm}
As $n\to\infty$, the sequence $(X_n(0), X_n^{(0)},U_n,L_n)$
LD converges in distribution in
 $\R\times\D(\R_+,\R)\times\D(\R_+,\D(\R_+,\R))^2$
  to $(X(0),X^{(0)},U,L)$\,.
\end{lemma}
\begin{proof}
Let
\begin{equation}
  \label{eq:8}
    X'^{(0)}_n(x,t)=\frac{1}{b_n\sqrt{n}}\,\sum_{i=1}^{\lfloor
    nt\rfloor}\,
(\ind{\eta^{(0)}_i>
  x}-(1-F_0(x)))
\end{equation}
and $X'^{(0)}_n=(X'^{(0)}_n(t)\,,t\in\R_+)$\,.
By the hypotheses of Theorem \ref{the:heavy_traffic},
$X_n(0)$ LD converges to $X(0)$\,, by Lemma \ref{le:kiefer} and $F_0$
being strictly increasing,
$X'^{(0)}_n$ LD converges to
$((\tilde K(F_0(x),t)\,,x\in\R_+)\,, t\in\R_+)$\,, where $\tilde K$ represents a
Kiefer idempotent process that is independent of $(X(0),K,B)$\,.
 Also
$(K_n,B_n)$ LD converges to $(K,B)$\,.
By  independence assumptions, these convergences hold jointly,
cf. Section \ref{sec:large-devi-conv}.
Since $Q_n(0)/n\to 1$ 
and $\hat A_n(t)/n\to\mu t$ super--exponentially in probability
by \eqref{eq:72'} and
 \eqref{eq:4}, respectively,
''Slutsky's theorem'' (Lemma \ref{the:sl}) yields  joint LD convergence
of $(X_n(0),X'^{(0)}_n,K_n,B_n,Q_n(0)/n,(\hat A_n(t)/n\,,t\in \R_+))$ to
$(X(0),((\tilde K(F_0(x),t)\,,x\in\R_+)\,, t\in\R_+),K,B,1,(\mu
t\,,t\in\R_+))$\,. It is also noteworthy that, by \eqref{eq:1a},
 \eqref{eq:63} and \eqref{eq:8},  
\begin{equation}
  \label{eq:207}
  X^{(0)}_n(t)=X'^{(0)}_n(t,\frac{Q_n(0)}{n}\wedge 1)
-(1-F_0(t))X_n(0)^-\,.
\end{equation}
In
order to deduce the LD convergence of $(X_n(0), X_n^{(0)},U_n,L_n)$
to $(X(0), X^{(0)},U,L)$\,,
it is left to recall \eqref{eq:47a}, 
\eqref{eq:80}, \eqref{eq:7}, note that, by Lemma \ref{le:idemp_prop},
 $(\tilde K(t,1)\,,t\in[0,1])$ is a Brownian bridge
idempotent process, and apply the continuous mapping principle,
 the associated composition mappings being
continuous at continuous limits.
(See Whitt \cite{Whi80} for more background on continuous functions in
the Skorohod space context.)
\end{proof}

\begin{lemma}
\label{le:exp_tight}  The sequence $\{\Theta_n\,,n\in\N\}$ is $\bbC$--exponentially tight of order
  $b_n^2$ in $\D(\R_+,\R)$\,.
\end{lemma}
\begin{proof}
      By Lemma \ref{le:sovm}, 
\eqref{eq:17}, \eqref{eq:20}  and the continuous mapping principle, 
 $J_n$ LD converges  to $J$\,, so, the sequence $J_n$ is
$\bbC$--exponentially tight. By \eqref{eq:41} it remains to check that
the sequence $M_n$ is $\bbC$--exponentially tight, which, according to
Theorem \ref{the:exp_tight}, is implied by the following convergences:
\begin{align}
  \lim_{r\to\infty}
\limsup_{n\to\infty}P(\sup_{s\in[0,t]}\abs{M_n(s)}>r)^{1/b_n^2}=0\notag
\intertext{ and }
\label{eq:28a}  \lim_{\delta\to0}
\limsup_{n\to\infty}\sup_{s\in[0,t]}P(\sup_{s'\in[0,\delta]}\abs{M_n(s'+s)-M_n(s)}>\epsilon)^{1/b_n^2}=0\,,
\end{align}
where $t>0$ and $\epsilon>0$\,.
The former convergence has already been proved, see \eqref{eq:23a}.
The proof of \eqref{eq:28a} proceeds along similar lines.
Since, with $\alpha\in\R$\,, the process $(e^{\alpha
  (M_n(s+s')-M_n(s))-(G_n(\alpha,s+s')-G_n(\alpha,s))}\,,s'\in\R_+)$ is a local martingale so that,
for arbitrary stopping time
$\tau$\,,
$E e^{\alpha
(  M_n(\tau+s)-M_n(s))-(G_n(\alpha,\tau+s)-G_n(\alpha,s))}\le1$\,,  by
Lemma 3.2.6 on p.282 in Puhalskii \cite{Puh01},
for arbitrary  $\gamma>0$\,, 
in analogy with \eqref{eq:56},
\[
    P(\sup_{s'\in[0,\delta]} (M_n(s')-M_n(s))\ge \epsilon)^{1/b_n^2}\le
e^{\alpha(\gamma- \epsilon)}
+P(\hat G_n(\alpha b_n^2,s+\delta)-\hat G_n(\alpha b_n^2,s)\ge\alpha
  b_n^2\gamma)^{1/b_n^2}\,.
\]
By \eqref{eq:57} and Lemma \ref{le:lln},
\begin{equation}
  \label{eq:58}
\frac{1}{b_n^2}\,  \hat G_n(\alpha b_n^2,t)\to\frac{\alpha^2}{2}\,\mu\int_0^t( t-u)
\frac{dF(u)}{1-F(u)}
\end{equation}
super--exponentially in probability at rate $b_n^2$\,. 
The latter  super--exponential convergence in probability 
being locally uniform in $t$\,, as the limit is a monotonic continuous
function starting at $0$\,, implies that, 
for  $\delta$ small enough, depending on $\alpha$\,,
\[
   \limsup_{n\to\infty}\sup_{s\in[0,t]} P(\sup_{s'\in[0,\delta]} 
(M_n(s'+s)-M_n(s))\ge \epsilon)^{1/b_n^2}\le
e^{\alpha(\gamma- \epsilon)}\,.
\]
Now, one chooses $\gamma<\epsilon$ and  sends $\alpha\to\infty$\,.
 \end{proof}

\begin{lemma}
    \label{le:joint_promezh}
As $n\to\infty$,
the sequence $(X_n(0), X_n^{(0)},
H_n,J_n,L_n,U_n,\Theta_n)$ LD converges in distribution in
$\R\times\D(\R_+,\R)^3\times\D(\R_+,\D(\R_+,\R))^2\times\D(\R_+,\R)$
to  $(X(0),X^{(0)},H,J,L,U,\Theta)$\,.
  \end{lemma}
  \begin{proof}
Let 
\begin{equation}
  \label{eq:22}
\Theta^{(l)}_n(t)=-\int_{\R_+^2}I_{t}^{(l)}(x,s)dU_n(x,s)\,,  
\end{equation}
where
\begin{equation}
  \label{eq:31}
  I_{t}^{(l)}(x,s)=\sum_{i=1}^\infty\ind{s\in(s^{(l)}_{i-1},s^{(l)}_i]}
\ind{x\in[0,t-s^{(l)}_{i-1}]}\,, 
\end{equation}
for some $0=s^{(l)}_0<s^{(l)}_1<\ldots$ such that
$s^{(l)}_i\to\infty$\,, as $i\to\infty$\,,
and $\sup_{i\ge1}(s^{(l)}_i-s^{(l)}_{i-1})\to0$\,, as $l\to\infty$\,, and, accordingly,
\[
\int_{\R_+^2}I_{t}^{(l)}(x,s)dU_n(x,s)
=\sum_{i=1}^\infty(U_n(t-s^{(l)}_{i-1},s^{(l)}_i)
-U_n(t-s^{(l)}_{i-1}) ,s^{(l)}_{i-1})\ind{s^{(l)}_{i-1}\le t}\,.
\]
Similarly, let
\begin{equation}
  \label{eq:34}
  \Theta^{(l)}(t)=-\int_{\R_+^2}I_{t}^{(l)}(x,s)\,dU(x,s)=
\sum_{i=1}^\infty(U(t-s^{(l)}_{i-1},s^{(l)}_i)
-U(t-s^{(l)}_{i-1},s^{(l)}_{i-1}))\ind{s^{(l)}_{i-1}\le t}\,.
\end{equation}
By the LD convergence of $Y_n$ to $Y$ in the hypotheses of Theorem
\ref{the:heavy_traffic},  
 Lemma \ref{le:sovm}, \eqref{eq:16}, \eqref{eq:47},
\eqref{eq:17}, \eqref{eq:20}, 
 and the continuous mapping principle, 
the sequence $(X_{n}(0), X_n^{(0)},
H_{n},J_{n},L_{n},U_{n})$ LD converges  to $(X(0), X^{(0)},
H,J,L,U)$\,.
Hence, the sequence $(X_{n}(0), X_n^{(0)},
H_{n},J_{n},L_{n},U_{n},\Theta_{n}^{(l)})$ LD converges to
$(X(0), X^{(0)},
H,J,L,U,\Theta^{(l)})$\,.
In addition, by Lemma \ref{le:exp_tight}, the sequence $\Theta_n$ is $\bbC$--exponentially
tight. Since the idempotent processes $X^{(0)},
H,J,L,U$  are seen to have continuous trajectories,
 the sequence $( X_n^{(0)},
H_n,J_n,L_n,U_n,\Theta_n)$ is $\bbC$--exponentially tight.
That a limit point of
$(X_n(0),X_n^{(0)},
H_n,J_n,L_n,U_n,\Theta_n)$ is distributed as
$(X(0), X^{(0)},
H,J,L,U,\Theta)$ would follow from the LD convergence of 
finite--dimensional distributions of
$(X_n(0), X_n^{(0)},
H_n,J_n,L_n,U_n,\Theta_{n})$ to  finite--dimensional distributions of
$(X(0), X^{(0)},
H,J,L,U,\Theta)$\,.
Owing to   \eqref{eq:26},
\begin{equation}
  \label{eq:213}
  \Theta^{(l)}(t)=-\int_{\R_+^2}I_{t}^{(l)}(x,s)\dot K(F(x),\mu s)\,dF(x)\,\mu ds\,.
\end{equation}
which implies, by \eqref{eq:12} and the
Cauchy--Schwarz inequality, that $\Theta^{(l)}\to \Theta$
  locally uniformly, as $l\to\infty$\,.
 Since  the sequence $(X_n(0), X_n^{(0)},
H_n,J_n,L_n,U_n)$ LD converges  to $(X(0), X^{(0)},
H,J,L,U)$\,, in order to prove the finite--dimensional LD convergence
 it suffices  to prove that
\begin{equation}
  \label{eq:35}
  \lim_{l\to\infty}\limsup_{n\to\infty}
P(\abs{\Theta_n^{(l)}(t)-\Theta_n(t)}>\epsilon)^{1/b_n^2}=0\,.
\end{equation}
 By \eqref{eq:19} and \eqref{eq:22},
\[
      \Theta_n(t)-\Theta_n^{(l)}(t)=\frac{1}{b_n\sqrt{n}}\sum_{i=1}^{\hat A_n(t)}
\sum_{j=1}^\infty
\ind{s^{(l)}_{j-1}\le t}\ind{\hat\tau_{n,i}\in(s^{(l)}_{j-1},s^{(l)}_j]}
\bl(\ind{\eta_i\in( t-\hat\tau_{n,i},t-s^{(l)}_{j-1}]}-
(F(t-s^{(l)}_{j-1})-F(t-\hat\tau_{n,i})\br)\,.
\]
Let 
\[
   \hat \Theta_n^{(l)}(s,t)=\frac{1}{b_n\sqrt{n}}\sum_{i=1}^{\hat A_n(s)}
\sum_{j=1}^\infty
\ind{s^{(l)}_{j-1}\le t}\ind{\hat\tau_{n,i}\in(s^{(l)}_{j-1},s^{(l)}_j]}
\bl(\ind{\eta_i\in( t-\hat\tau_{n,i},t-s^{(l)}_{j-1}]}-
(F(t-s^{(l)}_{j-1})-F(t-\hat\tau_{n,i}))\br)\,.
\]
Let $\mathcal{F}_{n}(s)$ represent the
 complete $\sigma$--algebra generated by the
random variables $\hat\tau_{n,j}\wedge \hat\tau_{n,\hat A_n(s)+1}$ and $\eta_{j\wedge
  \hat A_n(s)}$\,, where $j\in \N$\,. 
By part 4 of Lemma C.1 in Puhalskii and Reed  \cite{PuhRee10},
with $t$ held fixed, the 
process $(\hat\Theta_n^{(l)}(s,t)\,,s\in\R_+)$ is an
$\mathbb F_n$--locally square integrable martingale. Its measure of jumps is 
\begin{multline*}
  \mu_n^{(l)}([0,s],\Gamma)=
\sum_{i=1}^{\hat A_n(s)}
\sum_{j=1}^\infty
\ind{s^{(l)}_{j-1}\le
  t}\ind{\hat\tau_{n,i}\in(s^{(l)}_{j-1},s^{(l)}_j]}
\bl(\ind{(1-(F(t-s^{(l)}_{j-1})-F(t-\hat\tau_{n,i}))/(b_n\sqrt{n})\in\Gamma}
\ind{\eta_i\in( t-\hat\tau_{n,i},t-s^{(l)}_{j-1}]}\\
+\ind{(-(F(t-s^{(l)}_{j-1})-F(t-\hat\tau_{n,i}))/(b_n\sqrt{n})\in\Gamma}
\ind{\eta_i\not\in( t-\hat\tau_{n,i},t-s^{(l)}_{j-1}]}\br)\,,
\end{multline*}
where $\Gamma\subset \R\setminus\{0\}$\,.
Accordingly, the
 $\mathbb F_n$--predictable measure of jumps is
\begin{multline*}
  \nu_n^{(l)}([0,s],\Gamma)\\=
\sum_{i=1}^{\hat A_n(s)}
\sum_{j=1}^\infty
\ind{s^{(l)}_{j-1}\le
  t}\ind{\hat\tau_{n,i}\in(s^{(l)}_{j-1},s^{(l)}_j]}
\bl(\ind{(1-(F(t-s^{(l)}_{j-1})-F(t-\hat\tau_{n,i}))/(b_n\sqrt{n})\in\Gamma}
(F(t-s^{(l)}_{j-1})-F( t-\hat\tau_{n,i}))\\
+\ind{-(F(t-s^{(l)}_{j-1})-F(t-\hat\tau_{n,i}))/(b_n\sqrt{n})\in\Gamma}
(1-(F(t-s^{(l)}_{j-1})-F( t-\hat\tau_{n,i})))\br)\,.
\end{multline*}
For $\alpha\in\R$\,, as on p.214 in Puhalskii  \cite{Puh01},
 define the stochastic cumulant
\begin{multline}
      \label{eq:10}
      G^{(l)}_n(\alpha,s)=\int_0^s\int_\R(e^{\alpha x}-1-\alpha  x)\nu^{(l)}_n(ds',dx)\\=\sum_{i=1}^{\hat A_n(s)}
\sum_{j=1}^\infty\ind{s^{(l)}_{j-1}\le
  t}\ind{\hat\tau_{n,i}\in(s^{(l)}_{j-1},s^{(l)}_j]}\bl((e^{\alpha
  (1-(F(t-s^{(l)}_{j-1})-F(t-\hat\tau_{n,i})))/(b_n\sqrt{n})}\\
-1-
\frac{\alpha}{b_n\sqrt{n}}(1-(F(t-s^{(l)}_{j-1})-F(t-\hat\tau_{n,i}))))
(F(t-s^{(l)}_{j-1})-F(t-\hat\tau_{n,i}))\\+
(e^{-\alpha
  (F(t-s^{(l)}_{j-1})-F(t-\hat\tau_{n,i}))/( b_n\sqrt{n})}
-1+\frac{\alpha}
{b_n\sqrt{n}}(F(t-s^{(l)}_{j-1})-F(t-\hat\tau_{n,i})))
(1-(F(t-s^{(l)}_{j-1})-F(t-\hat\tau_{n,i})))\br)\,.
\end{multline}
The associated stochastic exponential is defined by 
\begin{equation}
  \label{eq:14}
  \mathcal{E}^{(l)}_n(\alpha,s)=e^{G^{(l)}_n(\alpha,s)}
\prod_{0<s'\le s}(1+\Delta G^{(l)}_n(\alpha,s'))e^{-\Delta G^{(l)}_n(\alpha,s')}\,,
\end{equation}
where $\Delta G^{(l)}_n(s')$ represents the jump of $
G^{(l)}_n(s'')$ with respect to $s''$ evaluated at $s''=s'$ and the
product is taken over the jumps.
By Lemma 4.1.1 on p.294 in
Puhalskii \cite{Puh01}, the process
$ \bl( \exp(\alpha b_n^2 \hat
  \Theta_n^{(l)}(s,t))\mathcal{E}^{(l)}_n(\alpha b_n^2,s)^{-1}\,,s\in\R_+\br)$
 is a well defined
 local martingale so that, for any stopping time $\tau$\,,
$  E\exp(\alpha b_n^2 \hat \Theta_n^{(l)}(\tau,t))
\mathcal{E}^{(l)}_n(\alpha b_n^2,\tau)^{-1}
\le1\,.
$ Lemma 3.2.6 on p.282 in Puhalskii \cite{Puh01} 
and \eqref{eq:14} imply that, for
$\alpha>0$ and $\gamma>0$\,,
\begin{equation}
  \label{eq:21}
    P(\sup_{s\in[0,t]} \hat \Theta_n^{(l)}(s,t)\ge \epsilon)
\le e^{\alpha b_n^2(\gamma-\epsilon)}+P   (
\mathcal{E}^{(l)}_n(\alpha b_n^2,t)\ge e^{\alpha
    b_n^2\gamma})\le e^{\alpha b_n^2(\gamma-\epsilon)}+P   (
G^{(l)}_n(\alpha b_n^2,t)\ge \alpha
    b_n^2\gamma)\,.
\end{equation}
By \eqref{eq:10},
\begin{multline*}
G^{(l)}_n(\alpha b_n^2,t)\le\hat A_n(t)\bl(\sup_{\abs{y}\le1}(e^{\alpha
  b_ny/\sqrt{n}}
-1-\alpha
\frac{b_ny}{\sqrt{n}})\sup_{j\in\N}
(F(t-s^{(l)}_{j-1})-F(t-s^{(l)}_j))\\+
\sup_{j\in\N}\sup_{y\in[F(t-s^{(l)}_{j}),F(t-s^{(l)}_{j-1})]}(e^{-\alpha
  b_n y/\sqrt{n}}
-1+\alpha
\frac{b_n}{\sqrt{n}}y)
\br)\,.
   \end{multline*}
As $n/b_n^2(e^{\alpha
  b_n y/\sqrt{n}}
-1-\alpha 
b_ny/\sqrt{n})\to\alpha^2y^2/2$ uniformly on bounded intervals, 
$\hat A_n(t)/n\to \mu t$
super--exponentially, as $n\to\infty$, and
$\sup_j(s^l_j-s^l_{j-1})\to0$\,, as $l\to\infty$\,, it follows that
\[
  \lim_{l\to\infty}\limsup_{n \to\infty}P   (G^{(l)}_n(\alpha b_n^2,t)\ge \alpha
    b_n^2\gamma)^{1/b_n^2}=0\,,
\]
which implies, thanks to \eqref{eq:21}, that
\[
\limsup_{l\to\infty}  \limsup_{n\to\infty}
  P(\sup_{s\in[0,t]}\hat \Theta_n^{(l)}(s,t)\ge \epsilon)^{1/b_n^2}
\le e^{\alpha (\gamma-\epsilon)}\,.
\]
Picking $\gamma<\epsilon$ and sending $\alpha$ to $\infty$ show
that the latter lefthand side equals zero.
A similar argument proves the convergence 
\[
\lim_{l\to\infty}    \limsup_{n\to\infty}
  P(\sup_{s\in[0,t]}(-\hat \Theta_n^{(l)}(s,t))\ge \epsilon)^{1/b_n^2}=0\,. 
\]
Noting that $\Theta_n(t)-\Theta_n^{(l)}(t)=\hat \Theta_n^{(l)}(t,t)$ yields 
the convergence \eqref{eq:35}.
  Lemma \ref{le:joint_promezh} has
been proved so that Theorem \ref{le:L} has been proved.

\end{proof}
Theorem \ref{le:L} has been proved. In order to obtain the assertion
of Theorem \ref{the:heavy_traffic},  note that $-K$ has
the same idempotent distribution as $K$
 and invoke the continuous mapping
principle in \eqref{eq:67}, which applies due to  Lemma B.2 in
Puhalskii and Reed \cite{PuhRee10}.
Part 2 of  Theorem \ref{the:heavy_traffic} holds because under its
hypotheses,
$Y(t)=\sigma W(t)-\beta\mu t$\,. The assertion of Theorem
\ref{the:ldp} follows on observing that $\exp(-I^Q(y))$, with
 $y\in\D(\R_+,\R)$\,, is the
deviability density of  the idempotent
  distribution of $X$ in the statement of Theorem 
\ref{the:heavy_traffic}. To see the latter, note that the mapping 
$(w^0,w,k)\to u$ specified by  \eqref{eq:1} is continuous when
restricted to the set $\{(w^0,w,k):\Pi^{W^0,W,K}(w^0,w,k)\ge a\}$\,,
where $\Pi^{W^0,W,K}(w^0,w,k)=\Pi^{W^0}(w^0)\Pi^{W}(w)\Pi^K(k)$ and $a\in(0,1]$\,, so that $X$ is  strictly Luzin 
on $\bl(\D([0,1],\R)\times\D(\R_+,\R)\times 
\D(\R_+,\D(\R_+,\R)),\Pi^{W^0,W,K}\br)$\,, see Appendix \ref{sec:large-devi-conv} for the
definition and properties of being strictly Luzin. Therefore,
\[
  \Pi^X(u)=\Pi^{W^0,W,K}(X=u)=
\sup_{(w^0,w,k):\,\eqref{eq:1}\text{ holds}}\Pi^{W^0}(w^0)\Pi^W(w)\Pi^K(k)\,.
\]

\section{Evaluating the deviation function}
\label{sec:eval-devi-funct}
This section concerns calculating the deviation function of Theorem
\ref{the:ldp}.
\begin{theorem}
  \label{the:var}Suppose that  $F$ is an absolutely continuous
  function and $I^Q(q)<\infty$\,. 
Then $q$ is absolutely continuous, 
$\dot q\in\mathbb L_2(\R_+)$\,, the infimum in \eqref{eq:2} is
attained uniquely and
\[
  I^Q(q)=\frac{1}{2}\,\int_0^\infty\overline p(t)\bl(\dot q(t)-\int_0^t\dot q(s)
\ind{q(s)>0}F'(t-s)\,ds
+(\beta-q_0^-) F_0'(t)\br)\,dt\,,
\]
where $\overline p(t)$ represents the  unique $ \mathbb L_2(\R_+)$--solution of the Fredholm equation of the
second kind
\begin{multline}
  \label{eq:114}
    (\mu+\sigma^2)p(t)=
  \dot q(t)-\int_0^t\dot q(s)
\ind{q(s)>0}F'(t-s)\,ds
+(\beta-x_0^-) F_0'(t)\\+F'_0(t)\int_0^\infty F_0'(s)p(s)\,ds
+\sigma^2\int_0^\infty F'(\abs{s-t})p(s)\,ds\\
+(\mu-\sigma^2)\int_0^\infty\int_0^{s\wedge t}F'(s-\tilde s)F'(t-\tilde
s)d\tilde s\,  p(s)ds\,,
\end{multline}
with $\dot q$\,, $F_0'$ and $F'$ representing  derivatives.
\end{theorem}
\begin{proof}
On writing
\begin{equation}
  \label{eq:52}
    \int_{\R_+^2} \ind{x+s\le t}\,\dot k(
F(x),\mu   s)\,dF(x)\,\mu\,ds=
  \int_0^{ t}\int_0^{F(t-s)}\dot k(x,\mu s)\,dx\,\mu\,ds\,,
\end{equation}
the equation  \eqref{eq:1} is seen to be of the form
\[
    q(t)=f(t)+\int_0^tq(t-s)^+\,dF(s)\,, t\in\R_+\,,
\]
with the functions $f(t)$ and $F(t)$ being absolutely continuous.
The function $q(t)$ is absolutely continuous by  Lemma \ref{le:ac}.
In addition, \eqref{eq:52} implies that, a.e.,
\begin{equation}
  \label{eq:119}
    \frac{d}{dt}\,\int_{\R_+^2} \ind{x+s\le t}\,{\dot k}(F(x),\mu s)\,dF(x)\,
\mu\,ds=\int_0^t{\dot k}(F(s),\mu (t-s))F'(s)\,\mu
ds\,.
\end{equation}

The infimum in \eqref{eq:2}
 is attained uniquely by coercitivity and strict convexity of the function
being minimised, cf., Proposition II.1.2 in  \cite{EkeTem76}.
Differentiation in \eqref{eq:1} with the account of \eqref{eq:119}
 obtains that, a.e.,
\begin{multline*}
\dot w^0(F_0(t))F_0'(t)
+\sigma\,\dot w(t)
-\int_0^tF'(t-s)\sigma\,\dot w(s)\,ds
+\int_0^t\dot k(F(s),\mu (t-s))F'(s)\,\mu
\,ds\\-\bl(  \dot q(t)-\int_0^t\dot q(s)\ind{q(s)>0}F'(t-s)\,ds+(\beta-x_0^-)
 F_0'(t)\br)=0\,.
\end{multline*}
In addition, the requirements that $w^0(0)=w^0(1)=0$ and 
$k(0,t)=k(1,t)=0$ give rise to the constraints
\begin{equation}
  \label{eq:38}
  \int_0^1\dot w^0(x)\,dx=0
\end{equation}
and
\begin{equation}
  \label{eq:39}
  \int_0^1\dot k(x,t)\,dx=0\,.
\end{equation}
Introduce
 the  map
\begin{multline*}
 \Phi:\, (\dot w^0,\dot w,\dot k)\to
\bl(\dot w^0(F_0(t))F_0'(t)
+\sigma\,\dot w(t)
-\int_0^tF'(t-s)\sigma\,\dot w(s)\,ds\\
+\int_0^t\dot k(F(s),\mu (t-s))F'(s)\,\mu\,
ds\,,t\in\R_+\br).
\end{multline*}
Since $F_0'(t)$ is bounded by \eqref{eq:73},
 $\Phi$ maps  
$V=\mathbb L_2([0,1])\times\mathbb L_2(\R_+)\times 
\mathbb L_2( [0,1]\times\R_+)$
to $\mathbb L_2(\R_+)$\,.
For instance, on using that $\int_0^\infty F'(s)\,ds=1$\,,
\[
  \int_0^\infty\bl(\int_0^tF'(t-s)\dot w(s)\,ds\br)^2\,dt
\le   \int_0^\infty\int_0^tF'(t-s)\dot w(s)^2\,ds\,dt=
\int_0^\infty\dot w(s)^2\,ds<\infty\]
and
\begin{multline*}
  \int_0^\infty\bl(\int_0^t\dot k(F(s),\mu (t-s))F'(s)\,\mu
ds\br)^2\,dt
\le \int_0^\infty\int_0^t\dot k(F(s),\mu (t-s))^2F'(s)\,\mu^2
ds\,dt\\=
\mu^2\int_0^\infty\int_0^1\dot k( x,t)^2dx\,
dt<\infty\,.
\end{multline*}
The method of Lagrange multipliers, more specifically, 
Proposition III.5.2 in 
\cite{EkeTem76} with $Y=\mathbb L_2(\R_+)^2\times \R$ and the set of
componentwise 
nonnegative functions as the cone $\mathcal{C}$\,, implies that
\begin{multline}
  \label{eq:15}
             I_{x_0}^Q(q)=\sup_{(p,\tilde p,r)\in \mathbb L_2(\R_+)^2\times\R}\;
\inf_{
\substack{(\dot w^0,\dot w,\dot k)\in
\mathbb L_2([0,1])\\\times\mathbb L_2( \R_+)\times
\mathbb L_2( [0,1]\times\R_+)}}\Bl(\frac{1}{2}\int_0^1\dot w^0(x)^2\,dx
+\frac{1}{2}\int_0^\infty \dot w(t)^2\,dt\\
+\frac{1}{2}\int_0^\infty\int_0^1 \dot k(x,t)^2\,dx\,dt
+\int_0^\infty p(t)\bl(\dot q(t)+F'(t)x_0^+
+(\beta-x_0^-) F_0'(t)\\-\int_0^t\dot q(s)
\ind{q(s)>0}F'(t-s)\,ds -\dot w^0(F_0(t))F_0'(t)
-\sigma\,\dot w(t)
+\int_0^tF'(t-s)\sigma\,\dot w(s)\,ds\\
-\int_0^t\dot k(F(s),\mu (t-s))F'(s)\,\mu
\,ds\br)\,dt
+r\int_0^1\dot w^0(x)\, dx
+\int_0^\infty\tilde p(t)\int_0^1\dot k(x,t)\,dx\,dt\Br)\,.
\end{multline}
Minimising in \eqref{eq:15} yields, with $(
\dot{\hat{ w}}^{0}(t)\,, 
\dot{\hat w}(t)\,, \dot{\hat{k}}(x,t))$ being optimal,
\begin{align*}
  \dot{\hat w}^0(x)-p(F_0^{-1}(x))
+r=0\,,\\
\dot{\hat w}(t)-\sigma
 p(t)+\sigma\int_0^\infty p(t+s)F'(s)\,ds=0\,,\\
\dot{\hat k}(x,t)-p(\frac{t}{\mu}+F^{-1}(x))
+\tilde p(t)=0\,.
\end{align*}
(For the latter, note that
\begin{multline*}
  \int_0^\infty p(t)\int_0^t\dot k(F(s),\mu (t-s))F'(s)\,\mu
ds\,dt=\int_0^\infty\int_s^\infty  p(t)\dot k(F(s),\mu (t-s))F'(s)\,\mu
\,dt\, ds\\=
\int_0^\infty\int_0^\infty  p(t+s)\dot k(F(s),\mu t)F'(s)\,\mu
\,dt\,ds
=
\int_0^\infty\int_0^1  p(\frac{t}{\mu}+F^{-1}(x))\dot k( x,t)\,dx\,dt\,.)
\end{multline*}
Hence,
\begin{multline*}
  I_{x_0}^Q(q)
=\sup_{(p,\tilde p,r)\in \mathbb L_2(\R_+)^2\times\R}
\bl(
\int_0^\infty p(t)\bl(\dot q(t)-\int_0^t\dot q(s)
\ind{q(s)>0}F'(t-s)\,ds
+(\beta-x_0^-) F_0'(t)\br)\,dt\\-\frac{1}{2}
\bl(\int_0^1(p(F_0^{-1}(x))
-r)^2\,dx
+\sigma^2\int_0^\infty( p(t)-\int_0^\infty p(t+s)F'(s)\,ds)^2\,dt
\\+\int_0^\infty\int_0^1 (p(\frac{t}{\mu}+F^{-1}(x))
-\tilde p(t))^2\,dx\,dt\br)\br)\,.
\end{multline*}
Given $p$\,,
the optimal $r$ is $\hat r=\int_0^1 p(F_0^{-1}(x))\,dx$ and the optimal
$\tilde p(t)$ is $\hat{\tilde p}(t)=\int_0^1p(t/\mu+F^{-1}(x))\,dx$\,.
(As a byproduct, $\hat w^0$ and $\hat k$ satisfy the constraints in
\eqref{eq:38} and \eqref{eq:39}.)
Therefore,
\begin{multline}
  \label{eq:42}
    I_{x_0}^Q(q)
=\sup_{p\in \mathbb L_2(\R_+)}
\bl(
\int_0^\infty p(t)\bl(\dot q(t)-\int_0^t\dot q(s)
\ind{q(s)>0}F'(t-s)\,ds
+(\beta-x_0^-) F_0'(t)\br)\,dt\\-\frac{1}{2}
\bl(\int_0^1(p(F_0^{-1}(x))
-\int_0^1 p(F_0^{-1}(\tilde x))\,d\tilde x)^2\,dx
+\sigma^2\int_0^\infty( p(t)-\int_0^\infty p(t+s)F'(s)\,ds)^2\,dt
\\+\int_0^\infty\int_0^1 (p(\frac{t}{\mu}+F^{-1}(x))
-\int_0^1p(\frac{t}{\mu}+F^{-1}(\tilde x))\,d\tilde x)^2\,dx\,dt\br)\br)\,.
\end{multline}
Note that
\begin{multline*} 
  \int_0^1(p(F_0^{-1}(x))
-\int_0^1 p(F_0^{-1}(\tilde x))\,d\tilde x)^2\,dx
=\int_0^1(p(F_0^{-1}(x))^2\,dx
-(\int_0^1 p(F_0^{-1}( x))\,d x)^2\\
=\int_0^\infty p(s)^2F_0'(s)\,ds
-(\int_0^\infty p(s)F'_0(s)\,d s)^2
\end{multline*}
and
\begin{multline*}
    \int_0^1 (p(\frac{t}{\mu}+F^{-1}(x))
-\int_0^1p(\frac{t}{\mu}+F^{-1}(\tilde x))\,d\tilde x)^2\,dx\\
=
\int_0^1p(\frac{t}{\mu}+F^{-1}( x))^2\,dx-
(\int_0^1p(\frac{t}{\mu}+F^{-1}( x))\,d x)^2\\
=
\int_0^\infty p(\frac{t}{\mu}+s)^2F'(s)\,ds-
(\int_0^\infty p(\frac{t}{\mu}+s)F'(s)\,d s)^2\,.
\end{multline*}
It follows that
\begin{multline*}
  I_{x_0}^Q(q)
=\sup_{p\in \mathbb L_2(\R_+)}
\bl(
\int_0^\infty p(t)\bl(\dot q(t)-\int_0^t\dot q(s)
\ind{q(s)>0}F'(t-s)\,ds
+(\beta-x_0^-) F_0'(t)\br)\,dt\\-\frac{1}{2}
\bl(\int_0^\infty p(s)^2F_0'(s)\,ds
-(\int_0^\infty p(s)F'_0(s)\,d s)^2
+\sigma^2\int_0^\infty( p(t)-\int_0^\infty p(t+s)F'(s)\,ds)^2\,dt
\\+\mu\int_0^\infty\int_0^\infty p(t+s)^2F'(s)\,ds\,dt-
\mu\int_0^\infty(\int_0^\infty p(t+s)F'(s)\,d s)^2\,dt\br)\br)\,.
\end{multline*}
 Observing that
\[
    \int_0^\infty p(s)^2F_0'(s)\,ds+
\mu\int_0^\infty\int_0^\infty
p(t+s)^2\,F'(s)ds\,dt
=\mu\int_0^\infty p(s)^2\,ds
\]
yields
\begin{multline}
  \label{eq:109}
  I_{x_0}^Q(q)=\sup_{p\in\mathbb L_2(\R_+)}\bl(
\int_0^\infty p(t)\bl(\dot q(t)-\int_0^t\dot q(s)
\ind{q(s)>0}F'(t-s)\,ds
+(\beta-x_0^-) F_0'(t)\br)\,dt\\-\frac{1}{2}
\bl(\mu\int_0^\infty p(s)^2\,ds-(\int_0^\infty p(s)F'_0(s)\,d s)^2
+\sigma^2\int_0^\infty(p(t)-\int_0^\infty p(t+s)F'(s)\,ds)^2\,dt\\
-
\mu\int_0^\infty(\int_0^\infty p(t+s)F'(s)\,d s)^2\,dt
\br)\br)\,.\end{multline}
By \eqref{eq:42}, the function in the $\sup$ is a strictly concave
function of $p$\,. 
Therefore,   a maximiser in \eqref{eq:109} is specified uniquely, see, e.g.,
Proposition II.1.2 in  \cite{EkeTem76}. The existence of the maximiser
follows from Proposition III.5.2 in 
\cite{EkeTem76}.

Varying $p$ in \eqref{eq:109} implies \eqref{eq:114}.
By the way, it is somewhat easier to carry out the calculations
if the righthand side of \eqref{eq:109} is rearranged as follows,
\begin{multline*}
I_{x_0}^Q(q)
  =\sup_{p\in\mathbb L_2(\R_+)}\bl(
\int_0^\infty p(t)\bl(\dot q(t)-\int_0^t\dot q(s)
\ind{q(s)>0}F'(t-s)\,ds
+(\beta-x_0^-) F_0'(t)\br)\,dt\\-\frac{1}{2}
\bl((\sigma^2+\mu)\int_0^\infty p(t)^2\,dt-(\int_0^\infty p(t)F'_0(t)\,d t)^2
-2\sigma^2\int_0^\infty p(t)\int_0^\infty p(t+s)F'(s)\,ds\,dt\\
+(\sigma^2-
\mu)\int_0^\infty(\int_0^\infty p(t+s)F'(s)\,d s)^2\,dt
\br)\br)\,.
\end{multline*}
 As the
maximiser in \eqref{eq:109} is unique, so is an $\mathbb L_2(\R_+)$--solution of the
Fredholm equation
\eqref{eq:114}. \end{proof}
\appendix

\section{Large deviation convergence and idempotent processes}
\label{sec:large-devi-conv}
This section reviews the basics of LD convergence and idempotent processes,
see, e.g., Puhalskii \cite{Puh01}.
Let $\mathbb E$ represent a metric space. 
 Let $\mathcal{P}(\mathbb{E})$ denote the
power set of $\mathbb{E}$. 
Set function $\Pi:\, \mathcal{P}(\mathbb{E})\to[0,1]$
is said to be a deviability if 
$\Pi(E)=\sup_{y\in E}\Pi(\{y\}),\,E\subset \mathbb{E},$
where the function $\Pi(y)=\Pi(\{y\})$ is such that 
$\sup_{y\in \mathbb{E}}\Pi(y)=1$ and the sets
$\{y\in \mathbb{E}:\,\Pi(y)\ge \gamma\}$ are compact for all $\gamma \in(0,1]$. 
(One can also refer to $\Pi$ as a maxi--measure or an idempotent probability.)
A deviability  is a tight set function in the sense that 
$\inf_{K\in\mathcal{K}(\mathbb E)}\Pi(\mathbb S\setminus K)=0$\,,
 where $\mathcal{K}(\mathbb E)$ stands for the collection of compact
subsets of $\mathbb E$\,.
If $\Xi$ is a directed set and $F_\xi\,, \xi\in \Xi\,,$ is a net of closed
subsets of $\mathbb E$ that is nonincreasing with respect to the
partial order on $\Xi$ by inclusion, then $\Pi(\cap_{\xi\in
  \Xi}F_\xi)=
\lim_{\xi\in \Xi}\Pi(F_\xi)$\,.
A property pertaining to elements of $\mathbb E$ is said to hold
$\Pi$--a.e. if the value of  $\Pi$ of the set of elements that do not
have this property equals 0\,.

Function $f$ from $\mathbb E$ to metric space 
$\mathbb E'$ is called an idempotent
variable.
 The idempotent distribution of the idempotent variable $f$ is defined
 as the set function $\Pi\circ
f^{-1}(\Gamma)=\Pi(f\in\Gamma),\,\Gamma\subset \mathbb E'$\,.
If $f$ is the canonical idempotent variable defined by
$f(y)=y$, then it has $\Pi$ 
as the idempotent distribution.  The continuous images of deviabilities are deviabilities, i.e., 
if $f:\,\mathbb E\to\mathbb E'$ is continuous, then $\Pi\circ f^{-1}$ defined by
$\Pi\circ f^{-1} (E')=\Pi( f^{-1}(E'))$ is a deviability on $\mathbb
E'$\,, where $E'\subset \mathbb E'$\,. Furthermore, this property
extends to the situation where $f:\,\mathbb E\to\mathbb E'$ is
strictly Luzin, i.e., continuous when restricted to the set
$\{y\in \mathbb E:\,\Pi(y)\ge a\}$\,, for arbitrary $a\in(0,1]$\,.
Thus, the idempotent distribution of a strictly Luzin idempotent
variables is a deviability. More generally, $f$ is said to be a Luzin
idempotent variable if the idempotent distribution of $f$ is a deviability. If
$f=(f_1,f_2)$\,, with $f_i$ assuming values in $\mathbb E'_i$\,,  
then the (marginal)
idempotent distribution of $f_1$ is defined by
$\Pi^{f_1}(y'_1)=\Pi(f_1=y'_1)=
\sup_{y:\,f_1(y)=y'_1}\Pi(y)$\,.
The  idempotent 
variables $f_1$ and $f_2$ are said to be independent if 
${\Pi}(f_1=y'_1,\,f_2=y'_2)=\Pi(f_1=y'_1)
\Pi(f_2=y'_2)$ for all $(y'_1,y'_2)\in
\mathbb E'_1\times\mathbb E'_2$\,,  
so, the joint distribution is the product of the
marginal ones.
Independence of finite collections of idempotent variables is defined
similarly.

A sequence $Q_n$
of probability measures on the Borel $\sigma$--algebra of $\mathbb
E$ is said to large
dviation (LD)  converge at rate $r_n$ to  deviability $\Pi$  if
for every bounded continuous non-negative function $f$ on $\mathbb{E}$
\[
\lim_{n\to\infty}\left(\int_{ \mathbb{E}}f(x)^{n}\,Q_n(dx)\right)^{1/r_n}=
\sup_{x\in \mathbb{E}}f(x)\Pi(x).
\]
Equivalently, one may require that 
$\lim_{n\to\infty}Q_n(\Gamma)^{1/r_n}=
\Pi( \Gamma)$  for every Borel set $\Gamma$ such that 
$\Pi$ of the interior of $\Gamma$ and $\Pi$ of the closure of $\Gamma$ agree.
If the $Q_n$ LD converge to $\Pi$\,, then
$\Pi(y)=\lim_{\delta\to0}\liminf_{n\to\infty}(Q_n(B_\delta(y))^{1/r_n}=
\lim_{\delta\to0}\limsup_{n\to\infty}(Q_n(B_\delta(y))^{1/r_n}$\,, for
all $y\in\mathbb E$\,, where $B_\delta(y)$ represents the open ball of
radius $\delta$ about $y$\,. (Closed balls may  be used as well.)
 The sequence $Q_n$ is said to be exponentially tight of order $r_n$ if 
$\inf_{K\in\mathcal{K}(\mathbb E)}\limsup_{n\to\infty}Q_n(\mathbb E\setminus
K)^{1/r_n}=0$\,. If the sequence $Q_n$ is exponentially tight of order
$r_n$\,, then there exists  subsequence $Q_{n'}$ that LD
converges at rate $r_{n'}$ to a deviability. Any such deviability will be
referred to as a Large Deviation (LD) limit point of the $Q_n$\,.
Given $\tilde E\subset \mathbb E$\,, the sequence $Q_n$ is said to be
$\tilde E$--exponentially tight if it is exponentially tight and
$\tilde\Pi(\mathbb E\setminus\tilde E)=0$\,, for any LD limit point of
$Q_n$\,. 

It is immediate that $\Pi$
is a deviability if and only if $I(x)=-\ln \Pi(x)$ is a tight  deviation
function, i.e., 
 the sets $\{x\in\mathbb E:\,I(x)\le \gamma\}$ are compact for all
$\gamma\ge0$ and $\inf_{x\in\mathbb E}I(x)=0$\,, and that the $Q_n$ LD
converge to $\Pi$ at rate $r_n$ if and only if they
 obey the LDP for rate $r_n$ with deviation function
$I$\,, i.e.,  $\liminf_{n\to\infty}(1/r_n)\,\ln Q_n(G)\ge -\inf_{x\in
  G}I(x)$\,, for all open sets $G$\,, and
$\limsup_{n\to\infty}(1/r_n)\,\ln Q_n(F)\le -\inf_{x\in
  F}I(x)$\,, for all closed sets $F$\,.

LD convergence of probability measures 
can be also expressed as LD convergence in distribution of the
associated random variables to idempotent variables. 
  Sequence
$\{X_n,\,n\in\N\}$  of random variables with  values in 
$\mathbb E' $ is said to 
LD converge in distribution at rate $r_n$ as $n\to\infty$
to   idempotent variable $X$ with values in $\mathbb E'$
if the sequence of the probability laws of the $X_n$ LD
converges to the idempotent distribution of $X$ at rate $r_n$. 
If random variables $X'_n$ and $X''_n$ are independent,
$X'_n$ LD converges to $X'$ and $X''_n$ LD converges to $X''$\,, then
the sequence $(X'_n,X''_n)$ LD converges to $(X',X'')$\,. If  sequence $\{{P}_n,\,n\in\N\}$
of probability measures  LD converges to  deviability
${\Pi}$\,, then one has LD convergence in
distribution of the canonical idempotent variables. A continuous
mapping principle holds: if random variables $X_n$ LD converge at rate
$r_n$ to idempotent variable $X$ and $f:\,\mathbb E'\to\mathbb E''$ is a continuous
function, then the random variables $f(X_n)$ LD converge at rate $r_n$
to $f(X)$\,, where $\mathbb E''$ is a metric space. 
The following version of Slutsky's theorem holds.
\begin{theorem}
  \label{the:sl}
Let $Y_n$ be random variables with values in $\mathbb
E''$
and let $a\in\mathbb E''$\,; if the sequence 
$X_n$ LD converges to $X$ 
and $Y_n\to a$ super--exponentially  in probability for rate 
 $r_n$\,, i.e., $P(d(Y_n,a)>\epsilon)^{1/r_n}\to0$\,, for
arbitrary $\epsilon>0$\,, then the sequence $(X_n,Y_n)$ LD converges
 at rate $r_n$
to $(X,a)$ in $\mathbb E'\times\mathbb E''$\,, where $d$ denotes the
metric on $\mathbb E''$\,.
\end{theorem}
Collection $(X_t,\,t\in\R_+)$ 
of idempotent variables on $\mathbb E$
is called an idempotent process.
The functions $(X_t(y),\,t\in\R_+)$ 
for various $y\in\mathbb E$ are
called trajectories (or paths) of $X$.
Idempotent processes are said to be independent if they are
independent as idempotent variables with values in the associated
function spaces. The concepts of idempotent processes with independent
and (or) stationary increments mimic those for stochastic processes.

Since this paper deals with stochastic processes having
rightcontinuous trajectories with lefthand limits, the underlying
 space $\mathbb E$ may be assumed to be a
Skorohod space $\D(\R_+,\R^m)$\,, for suitable $m$\,.

Suppose, $X_n=(X_n(t)\,,t\in\R_+)$ is a sequence of stochastic processes which
assumes values in a metric space $E'$ with metric $d'$
 and has rightcontinuous trajectories
with lefthand limits. The sequence $X_n$ is said to be exponentially
tight of order $r_n$ if the sequence of the distributions of $X_n$ as
measures on the Skorohod space $\D(\R_+,\mathbb E')$ is exponentially tight of
order $r_n$\,. It is said to be $\bbC$--exponentially tight if any LD
limit point is the law of an idempotent process with continuous
trajectories, i.e.,  $\Pi(\D(\R_+,\mathbb E')
\setminus\bbC(\R_+,\mathbb E'))=0$\,, whenever 
 $\Pi$ is an LD limit point of the laws of
$X_n$\,.

The method of finite--dimensional distributions for LD convergence of
stochastic processes is summarised in the next theorem,
Puhalskii \cite{Puh91}. The proof
mimics that in weak convergence theory.
\begin{theorem}
  \label{the:fd} If, for all tuples 
$t_1< t_2<\ldots < t_l$ with the $t_i$ coming from a dense subset of
$\R_+$\,,
the sequence of
 vectors $(X_n(t_1),\ldots,X_n(t_l))$ LD converges  in $\R^l$ at rate $r_n$
to $(X(t_1),\ldots,X(t_l))$ and the sequence $X_n$ is
$\bbC$--exponentially tight of order $r_n$\,, then
$X$ is a continuous path idempotent process and the sequence
$X_n$ LD converges in $\D(\R_+,\mathbb E')$ at rate $r_n$ to $X$\,.
\end{theorem}
\begin{theorem}
\label{the:exp_tight}
Suppose $\mathbb E'$ is, in addition, complete and separable.
The sequence $X_n$ is $\bbC$--exponentially tight of order $r_n$
if and only if
  \begin{enumerate}
  \item the sequence $X_n(t)$ is exponentially tight of order $r_n$\,,
    for all $t$ from a dense subset of $\R_+$\,, and,
\item 
for all $\epsilon>0$ and $L>0$\,,
  \begin{equation}
    \label{eq:83}
    \lim_{\delta\to0}\limsup_{n\to\infty}\sup_{t\in[0,L]}
      P(\sup_{s\in[0,\delta]}d'(X_n(t+s),X_n(t))>\epsilon)^{1/r_n}=0\,.
  \end{equation}
  \end{enumerate}
\end{theorem}
\begin{proof}
  The necessity of the condition follows from the continuity of the
  projection mapping and the continuous mapping principle. 
 The proof of sufficiency is standard and is included for
  completeness.
Let, for $L>0$\,, $\delta>0$ and $X$ from the Skorohod space
$\D(\R_+,\mathbb E')$\,,
\begin{equation}
  \label{eq:206}
  w_L(X,\delta)=\sup_{t,s\in[0,L]:\,\abs{t-s}\le\delta} d'(X(t),X(s))\,.
\end{equation}
Since
\begin{multline*}
  P(w_L(X_n,\delta)>\epsilon)\le P(\bigcup_{i=1}^{\lfloor
    L/\delta\rfloor
    +1}\{3\sup_{t\in[i\delta,(i+1)\delta]}d'(X_n(t),X_n(i\delta))>\epsilon\}\\
\le \sum_{i=1}^{\lfloor
    L/\delta\rfloor
    +1}P(3\sup_{t\in[i\delta,(i+1)\delta]}d'(X_n(t),X_n(i\delta))>\epsilon)
\le(\lfloor\frac{
    L}{\delta}\rfloor
    +1)\sup_{t\in[0,L]} 
P(\sup_{s\in[t,t+\delta]}d'(X_n(s),X_n(t))>\frac{\epsilon}{3})\,,
\end{multline*}
the hypotheses imply that
\begin{equation}
  \label{eq:208}
 \lim_{\delta\to0}\limsup_{n\to\infty}P(w_L(X_n,\delta)>\epsilon)^{1/r_n}=0\,.
\end{equation}
Let also
\begin{equation}
  \label{eq:210}
  w'_L(X,\delta)=\inf_{\substack{0=t_0<t_1<\ldots<t_k=L:\\
t_j-t_{j-1}>\delta}}\,\,
\max_{j=1,\ldots,k}
\sup_{u,v\in [t_{j-1},t_j)}d'(X(u),X(v))\,.
\end{equation}
Since $w'_L(X,\delta)\le w_L(X,2\delta)$\,, provided $\delta< L/2$\,,
by \eqref{eq:208},
\begin{equation}
  \label{eq:211}
 \lim_{\delta\to0}\limsup_{n\to\infty}P(w_L'(X_n,\delta)>\epsilon)^{1/r_n}=0\,.  
\end{equation}
As each $X_n$ is  a member of the Skorohod space $\D(\R_+,\mathbb
E')$,
\[
  \lim_{\delta\to0}P(w_L'(X_n,\delta)>\epsilon)=0\,,
\]
so, on recalling \eqref{eq:211},
\begin{equation}
  \label{eq:209}
  \lim_{\delta\to0}\sup_{n}P(w'_L(X_n,\delta)>\epsilon)^{1/r_n}=0\,.
\end{equation}
Let $\{t_1, t_2,\ldots\}$ represent a dense subset of $\R_+$ such that 
$X_n(t_1),X_n(t_2),\ldots$ are exponentially tight of order $r_n$\,.
Since $\mathbb E'$ is complete and separable, every probability
measure on $\mathbb E'$ is tight, so, it may be assumed that
there exist
  compact subsets $K_1,K_2,\ldots$ such that, for all
$n$\,,
\begin{equation}
  \label{eq:200}
  P(X_n(t_i)\notin K_i)^{1/r_n}<\frac{1}{2^i}\,, i\in\N\,.
\end{equation}
Choose positive $\delta_1,\delta_2,\ldots$ such that
\begin{equation}
  \label{eq:201}
  P(w_{i}'(X_n,\delta_i)>\frac{1}{2^i})^{1/r_n}<\frac{\epsilon}{2^i}\,, i\in\N\,.
\end{equation}
The set
\begin{equation}
  \label{eq:202}
  A=\bigcap_{i\in \N}\{X:\,w'_{i}(X,\delta_i)\le\frac{1}{2^i}\,,
X(t_i)\in K_i\}
\end{equation}
has a compact closure, see, e.g.,
 Theorem A2.2 on p. 563 in Kallenberg \cite{Kal02}, and
\begin{equation}
  \label{eq:205}
  P(X_n\notin A)^{1/r_n}\le
\sum_{i=1}^\infty P(w_{i}(X_n,\delta_i)>\frac{1}{2^i})^{1/r_n}
+\sum_{i=1}^\infty P(X_n(t_i)\notin K_i)^{1/r_n}<2\epsilon\,.
\end{equation}
Thus, the sequence $X_n$ is exponentially tight of order $r_n$ in
$\D(\R_+,\mathbb E')$\,. Let $\Pi$ represent a deviability on 
$\D(\R_+,\mathbb E')$ that is an LD limit point of the distributions
of $X_n$\,. It is proved next that $\Pi(X)=0$ if $X$ is  a
discontinuous function. Suppose $X$ has a jump at $t$\,, i.e.,
$d'(X(t),X(t-))>0$\,.
Let $\rho$ denote a metric in $\D(\R_+,\mathbb E')$\,.
It may be assumed that
if $\rho(X',X)<\delta$\,, then there exists a continuous nondecreasing
function $\lambda(t)$ such that $\sup_{s\le t+1}d'(X'(\lambda
(s)),X(s))< \delta$ and $\sup_{s\le t+1}\abs{\lambda
  (s)-s}<\delta$\,, see, e.g., Ethier and Kurtz \cite{EthKur86},
Jacod and Shiryaev \cite{jacshir}.
Then, assuming that $2\delta<t$ and $\delta<1$\,,
 $\inf_{s'\in[s-\delta,s+\delta]}d'(X(s'),X(s))< \delta$\,, where $s\in[0,t+1]$\,.
Note that
\begin{multline*}
  d'(X(t),X(t-))\le d'(X(t),X'(s_1))+d(X'(s_1),X'(t))+
d'(X'(t),X'(s_2))+d'(X'(s_2),X(s_3))\\+d'(X(s_3),X(t-))
\end{multline*}
so that, with $s_1\in[t-\delta,t+\delta]$ such that $d'(X'(s_1),X(t))<\delta$\,,  $s_3\in[t-\delta,t]$ such that $d'(X(s_3),X(t-))<\delta$\,, and
 $s_2\in[s_3-\delta,s_3+\delta]$ such that
$d'(X'(s_2),X(s_3))<\delta$\,,
\[
  d'(X(t),X(t-))< 3\delta
+2\sup_{s\in[t-2\delta,t+\delta]}d(X'(s),X'(t))\,,
\]
which implies that, for $\delta$ small enough,
\begin{equation}
  \label{eq:212}
  \sup_{s\in[t-2\delta,t+\delta]}d'(X'(s),X'(t))>
\frac{d'(X(t),X(t-))}{3}\,.
\end{equation}
Since
\begin{multline*}
    \sup_{s\in[t-2\delta,t+\delta]}d'(X'(s),X'(t))\le
    \sup_{s\in[t,t+\delta]}d'(X'(s),X'(t)) +
  \sup_{s\in[t-2\delta,t-\delta]}d'(X'(s),X'(t))\\
+\sup_{s\in[t-\delta,t]}d'(X'(s),X'(t))
\le     \sup_{s\in[t,t+\delta]}d'(X'(s),X'(t)) +
  \sup_{s\in[t-2\delta,t-\delta]}d'(X'(s),X'(t-2\delta))\\
+\sup_{s\in[t-\delta,t]}d'(X'(s),X'(t-\delta))
+d'(X'(t-\delta),X'(t-2\delta))
+2d'(X'(t),X'(t-\delta))
\,,
\end{multline*}
 for  $\epsilon=d'(X(t),X(t-))/18$\,,
\begin{multline*}
  P(\rho(X_n,X)<\delta)\le
P(\sup_{s\in[t,t+\delta]}d'(X_n(s),X_n(t))>\epsilon)+
  P(\sup_{s\in[t-2\delta,t-\delta]}d'(X_n(s),X_n(t-2\delta))>\epsilon)\\+
P(\sup_{s\in[t-\delta,t]}d'(X_n(s),X_n(t-\delta))>\epsilon)\,.
\end{multline*}
Therefore, assuming $L\ge t$ and $r_n\ge1$\,,
\begin{multline*}
  P(\rho(X_n,X)<\delta)^{1/r_n}
\le P(
\sup_{s\in[t,t+\delta]}d'(X_n(s),X_n(t))>\epsilon)^{1/r_n}
+P(
\sup_{s\in[t-\delta,t]}d'(X_n(s),X_n(t-\delta))>\epsilon)^{1/r_n}
\\+P(\sup_{s\in[t-2\delta,t-\delta]}
d'(X_n(t),X_n(t-2\delta))>\epsilon)^{1/r_n}
\le3\sup_{t'\in[0,L]}
      P(\sup_{s\in[0,\delta]}d'(X_n(t'+s),X_n(t'))>\epsilon)^{1/r_n}\,.
\end{multline*}
Since\[
    \Pi(X)=\lim_{\delta\to0}\liminf_{n\to\infty}P(\rho(X_n,X)<\delta)^{1/r_n}\,,
\]
\eqref{eq:83} implies that $\Pi(X)=0$\,.
 \end{proof}
The discussion below concerns the properties of
 the idempotent processes that feature prominently in the paper.
All the processes assume values in $\R$\,.
The standard  Wiener idempotent process denoted by  $W=(W(t)\,,t\in\R_+)$ 
is defined as an idempotent process
with idempotent distribution
\begin{equation}
  \label{eq:59}
  \Pi^{W}(w)=\exp\bl(-\frac{1}{2}\int_0^\infty\dot w(t)^2\,dt\br)
\end{equation}
provided $w=(w(t)\,,t\in\R_+)\in\D(\R_+,\R)$
is absolutely continuous and $w(0)=0$, and
$\Pi^{W}(w)=0$ otherwise. The restriction to $[0,t]$ produces a
standard Wiener idempotent process on $[0,t]$ which is specified by the
idempotent distribution $
  \Pi_t^{W}(w)=\exp\bl(-1/2\,\int_0^t\dot w(s)^2\,ds\br)$\,.
The  Brownian bridge idempotent process on $[0,1]$ denoted by 
$W^0=(W^0(x),\,x\in[0,1])$ is defined as an
  idempotent process with the idempotent distribution
\begin{equation}
  \label{eq:10}
  \Pi^{W^0}(w^0)=\exp\bl(-\frac{1}{2}\int_0^1\dot w^0(x)^2\,dx\br)
\end{equation}
provided $w^0=(w^0(x)\,,x\in[0,1])\in\D([0,1],\R)$ is absolutely continuous and $w^0(0)=w^0(1)=0$, and
$\Pi^{W^0}(w^0)=0$ otherwise.
The  Brownian sheet idempotent process on $ [0,1]\times\R_+$ denoted by $(B(x,t)\,,x\in[0,1],\,t\in\R_+)$ is defined as a
two--parameter idempotent process with the distribution
\begin{equation}
  \label{eq:60}
  \Pi^B(b)=
\exp\bl(-\frac{1}{2}\int_{[0,1]\times\R_+}
\dot b(x,t)^2\,dx\,dt\br)
\end{equation}
provided $b=(b(x,t)\,,x\in[0,1]\,, t\in\R_+)$ is absolutely continuous with respect to Lebesgue measure on
$\R_+\times [0,1]$ and $b(x,0)=b(0,t)=0$, and
$\Pi^{B}(b)=0$ otherwise.
 The  Kiefer idempotent process on $[0,1]\times\R_+$ denoted by
$(K(x,t)\,,x\in[0,1]\,,t\in\R_+),$ is defined
as a two--parameter idempotent process with the idempotent distribution
\begin{equation}
  \label{eq:11}
  \Pi^K(k)=\exp\bl(-\frac{1}{2}\int_{[0,1]\times\R_+}
\dot k(x,t)^2\,dx\,dt\br)
\end{equation}
provided $k=(k(x,t)\,,x\in[0,1]\,, t\in\R_+)$ is absolutely continuous with respect to Lebesgue measure on
$\R_+\times [0,1]$ and $k(0,t)=k(1,t)=k(x,0)=0$, and
$\Pi^{K}(k)=0$ otherwise.
In this paper we view it as an element of $\D(\R_+,\D([0,1],\R_+))$\,.
Furthermore, as the deviabilities  that
$W$\,, $W^0$\,, $B$, or $K$ have discontinuous paths are equal to zero, these
idempotent processes can be considered as having paths from
$\bbC(\R_+,\R)$\,, $\bbC([0,1],\R)$\,, $\bbC(\R_+,\bbC([0,1],\R))$\,
and
$\bbC(\R_+,\bbC([0,1],\R))$\,, respectively.

Being LD limits of their stochastic prototypes, these idempotent
processes have similar properties, as  summarised in the next lemma.
\begin{lemma}
  \label{le:idemp_prop}
  \begin{enumerate}
  \item For $x>0$\,, 
the idempotent process $(B(x,t)/\sqrt{x}\,,t\in\R_+)$ is a
    standard Wiener idempotent process.
\item 
For $t>0$\,, $(K(x,t)/\sqrt{t}\,,x\in[0,1])$ is a Brownian bridge
idempotent process. For $x\in(0,1)$\,, the idempotent process
  $(K(x,t)/\sqrt{x(1-x)}\,,t\in\R_+)$ is a standard Wiener idempotent
  process.
\item The Kiefer idempotent process can be written as
  \begin{equation}
    \label{eq:61}
    K(x,t)=-\int_0^x\frac{K(y,t)}{1-y}\,dy+B(x,t)\,,
  \end{equation}
where $B(x,t)$ is a Brownian sheet idempotent process. Conversely, if
$B(x,t)$ is a Brownian sheet idempotent process and \eqref{eq:61}
holds, then $K(x,t)$ is a Kiefer idempotent process.
Also,
\begin{equation}
  \label{eq:28}
    W^0(x)=-\int_0^x\frac{W^0(y)}{1-y}\,dy+W'(x)\,,
\end{equation}
where $W'(x)$ is a standard Wiener idempotent process.
   \end{enumerate}
\end{lemma}
\begin{proof}
  Parts 1 and 2 are elementary. For instance, 
  \begin{multline*}
    \Pi((\frac{K(x,t)}{\sqrt{t}}\,,x\in[0,1])=(w^0(x)\,,x\in[0,1]))\\=
\sup_{k:\,k(x,t)=\sqrt{t}w^0(x)\,,x\in[0,1]}
\exp\bl(-\frac{1}{2}\int_{\R_+}\int_0^1
\dot k(x,s)^2\,dx\,ds\br)\,.
  \end{multline*}
An application of the Cauchy--Schwarz inequality shows that
the optimal $k$ is 
\[
  k(x,s)=\frac{s\wedge t}{\sqrt{t}}\,w^0(x)\,.
\]
As for $(K(x,t)/\sqrt{x(1-x)}\,,t\in\R_+)$\,, the optimal trajectory 
$(k(y,t)\,,y\in[0,1])$ to get to $w(t)\sqrt{x(1-x)}$ at $x$ is
\[
  k(y,t)=
  \begin{cases}
    w(t)\sqrt{x(1-x)}\,\displaystyle\frac{y}{x}\,,&\text{ if
    }y\in[0,x]\,,
\\\\
    w(t)\sqrt{x(1-x)}\,\displaystyle\frac{1-y}{1-x}\,,&\text{ if
    }y\in[x,1]\,.
  \end{cases}
\]
In order to prove part 3, it suffices to show that if
\begin{equation}
  \label{eq:62}  
    k(x,t)=-\int_0^x\frac{k(y,t)}{1-y}\,dy+b(x,t)\,,
\end{equation}
with $k$ and $b$ being absolutely continuous and with $\Pi^B(b)>0$\,, then
$k(1,t)=0$ and
\begin{equation}
  \label{eq:64}
  \int_0^\infty\int_0^1 \dot k(x,t)^2\,dx\,dt 
= \int_0^\infty\int_0^1 \dot b(x,t)^2\,dx\,dt \,.
\end{equation}
Solving \eqref{eq:62} yields, for $x<1$\,,
\begin{equation}
  \label{eq:65}
  k(x,t)=(1-x)\int_0^x\frac{b_y(y,t)}{1-y}\,dy\,,
\end{equation}
where $b_y(y,t)=\int_0^t \dot b(s,y)\,ds$\,.
By the Cauchy--Schwarz inequality,
\begin{multline}
  \label{eq:66}
  \abs{\int_0^x\frac{b_y(y,t)}{1-y}\,dy}
\le\sqrt{\int_0^x\frac{1}{(1-y)^2}\,dy}\sqrt{\int_0^1b_y(y,t)^2dy}
\le\frac{\sqrt{t}}{\sqrt{1-x}}\,\sqrt{\int_{ [0,1]\times \R_+}
\dot b(y,s)^2\,dy\,ds}\,.
\end{multline}
We obtain that $k(x,t)\to0$ as $x\to1$ provided $\Pi^B(b)>0$\,.

Let $k_t(x,t)=\int_0^x\dot k(y,t)\,dy$\,. We show that if
$\Pi^B(b)>0$\,, then, a.e.,
\begin{equation}
  \label{eq:68}
\frac{1}{1-x}\,  k_t(x,t)^2 \to0\text{ as } x\to 1\,.
\end{equation}
Given arbitrary $a>0$\,, by \eqref{eq:65},
\begin{multline*}
  \frac{k_t(x,t)^2}{1-x}=(1-x)\Bl(\int_0^x\frac{\dot
    b(y,t)\,dy}{1-y}\Br)^2\le 2(1-x)a^2
\Bl(\int_0^x\frac{dy}{1-y}\Br)^2\\
+2(1-x)\int_0^x\frac{dy}{(1-y)^2}
\int_0^x\dot b(y,t)^2\ind{\abs{\dot b(y,t)}> a}\,dy
\le 2(1-x)\abs{\ln(1-x)}^2a^2
+2\int_0^x\dot b(y,t)^2\ind{\abs{\dot b(y,t)}> a}\,dy\,.
\end{multline*}
Since $\int_0^1\dot b(y,t)^2\,dy<\infty$ a.e., the latter rightmost
side tends to $0$ a.e.\,, as $x\to1$ and $a\to\infty$\,.

We are now in a position to prove \eqref{eq:64}.
Owing to \eqref{eq:62},
\begin{equation}
  \label{eq:69}
  \int_0^\infty \int_0^1\dot b(x,t)^2\,dx\,dt=
\int_0^\infty \int_0^1\Bl(     \dot k(x,t)^2+2\dot k(x,t)\frac{k_t(x,t)}{1-x}+
(\frac{k_t(x,t)}{1-x})^2\Br)\,dx\,dt\,.
\end{equation}
Integration by parts with the account of \eqref{eq:68} yields, for
almost all $t$\,,
\[
  \int_0^1\dot k(x,t)\frac{k_t(x,t)}{1-x}\,dx
=-  \int_0^1 k_t(x,t)\bl(\frac{\dot k(x,t)}{1-x}
+\frac{ k_t(x,t)}{(1-x)^2})\,dx
\]
so that
\[
  \int_0^1\dot k(x,t)\frac{k_t(x,t)}{1-x}\,dx=-\frac{1}{2}\,
\int_0^1\frac{ k_t(x,t)^2}{(1-x)^2}\,dx\,.
\]
Recalling \eqref{eq:69} implies \eqref{eq:64}.\end{proof}

 \section{A nonlinear renewal equation}\label{renewal}
This section is concerned with the  properties of the equation
\begin{equation}
  \label{eq:100}
  g(t)=f(t)+\int_0^tg(t-s)^+\,dF(s)\,, t\in\R_+\,.
\end{equation}
It is assumed that $f(t)$ is a locally bounded measurable
function and
that $F(t)$ is a continuous distribution function 
 on $\R_+$ with $F(0)<1$\,.
The existence and uniqueness  of an essentially locally bounded
solution $g(t)$ to \eqref{eq:100} follows from Lemma B.2 in Puhalskii and
Reed \cite{PuhRee10}.
\begin{lemma}
  \label{le:ac} If the functions $f$ and $F$ are absolutely continuous with
respect to Lebesgue measure and $F(0)=0$\,, then the function $g$ is absolutely
continuous too.
\end{lemma}
\begin{proof}
  Use Picard iterations. Let $g_0(t)=f(t)$ and
  \begin{equation}
    \label{eq:160}
    g_{k}(t)=f(t)+\int_0^tg_{k-1}(t-s)^+\,dF(s)\,.
  \end{equation}
The functions $g_k$ are seen to be continuous. Let $\epsilon>0$\,, 
$T>0$ and $0\le t_0\le t_1\le\ldots\le t_l\le T$\,.
Since $g_k\to g$ locally uniformly, see Lemma B.1 in
Puhalskii and Reed \cite{PuhRee10}, 
the function $g$ is continuous and $\sup_k\sup_{t\in[0,T]}\abs{g_k(t)}\le M$\,, for some $M>0$\,.
  Note that
\begin{multline}
  \label{eq:162}
  \abs{\int_0^{t_i}g_{k-1}(t_i-s)^+\,dF(s)-\int_0^{t_{i-1}}g_{k-1}(t_{i-1}-s)^+\,dF(s)}
\\\le \int_{t_{i-1}}^{t_i}\abs{g_{k-1}(t_i-s)}\,dF(s)
+\int_0^{T}\ind{s\le t_{i-1}}\abs{g_{k-1}(t_i-s)-g_{k-1}(t_{i-1}-s)}\,dF(s)\,.
\end{multline}
Let
\begin{equation}
  \label{eq:180}
  \psi(\delta)=\sup_{0\le t_0\le t_1\le\ldots\le t_l\le T}
\{\sum_{i=1}^l \abs{f(t_i)-f(t_{i-1}}+
M\sum_{i=1}^l \abs{F(t_i)-F(t_{i-1})}:\,\sum_{l=1}^l(t_i-t_{i-1})\le\delta\}
\end{equation}
and
\begin{equation}
  \label{eq:186}
    \phi_k(\delta)=\sup_{0\le t_0\le t_1\le\ldots\le t_l\le T}
\{\sum_{i=1}^l \abs{g_k(t_i)-g_k(t_{i-1})}:\,\sum_{l=1}^l(t_i-t_{i-1})\le\delta\}\,.
\end{equation}
By \eqref{eq:160} and \eqref{eq:162}, for $k\ge1$\,,
\begin{equation}
  \label{eq:187}
  \phi_k(\delta)\le \psi(\delta)+\phi_{k-1}(\delta)F(T)\,.
\end{equation}
Let
\begin{equation}
  \label{eq:188}
      \phi(\delta)=\sup_{0\le t_0\le t_1\le\ldots\le t_l\le T}
\{\abs{g(t_i)-g(t_{i-1})}:\,\sum_{l=1}^l(t_i-t_{i-1})\le\delta\}\,.
\end{equation}
Suppose that $F(t_0)<1$\,. Since $g_k\to g$ locally uniformly, as $k\to\infty$\,,
$\phi_k(\delta)\to\phi(\delta)$\,.
Letting $k\to\infty$ in  \eqref{eq:187} implies that $\phi(\delta)\le
\psi(\delta)/(1-F(t_0))$ so that $\phi(\delta)\to0$\,, as
$\delta\to0$\,. Hence, $g(t)$ is absolutely continuous on $[0,t_0]$\,.
Next, as in Puhalskii and Reed \cite{PuhRee10}, write, for
$t\in[0,t_0]$\,,
\[
  g(t+t_0)=f(t+t_0)+\int_t^{t+t_0}g(t+t_0-s)\,dF(s)+
\int_0^tg(t+t_0-s)\,d F(s)\,.
\]
By what has been proved, the sum of the first two terms on the
righthand side is an absolutely continuous function of $t$ on
$[0,t_0]$\,. The preceding argument implies that $g(t+t_0)$ is
absolutely continuous in $t$ on $[0,t_0]$\,. Iterating the argument
proves the absolute continuity of $g(t)$ on $\R_+$\,.
  \end{proof}
\def\cprime{$'$} \def\cprime{$'$} \def\cprime{$'$} \def\cprime{$'$}
  \def\cprime{$'$} \def\cprime{$'$} \def\cprime{$'$} \def\cprime{$'$}
  \def\cprime{$'$} \def\cprime{$'$} \def\cprime{$'$}
  \def\polhk#1{\setbox0=\hbox{#1}{\ooalign{\hidewidth
  \lower1.5ex\hbox{`}\hidewidth\crcr\unhbox0}}} \def\cprime{$'$}
  \def\cprime{$'$} \def\cprime{$'$}


\end{document}